\documentclass[a4paper]{amsart}
\usepackage[all]{xy}
\usepackage[colorlinks=true]{hyperref}
\usepackage{enumerate}
\usepackage{amssymb}
\usepackage{mathtools}
\usepackage{comment}
\usepackage{float}
\usepackage{color}
\newtheorem{thm}{Theorem}[section]
\newtheorem*{thmx}{Theorem} % not numbered
\newtheorem{prop}[thm]{Proposition}
\newtheorem{lem}[thm]{Lemma}
\newtheorem{cor}[thm]{Corollary}

\theoremstyle{definition}
\newtheorem{defn}[thm]{Definition}
\newtheorem{ex}[thm]{Example}
\theoremstyle{remark}
\newtheorem{rem}[thm]{Remark}

\newcommand{\Rr}{\mathbb R}

\newcommand{\Zz}{\mathbb Z}
\newcommand{\Nn}{\mathbb N}
\newcommand{\Cc}{\mathbb C}

\newcommand{\Kk}{\mathbb K}

\newcommand{\g}{\mathfrak{g}}
\newcommand{\h}{\mathfrak{h}}

\newcommand{\set}[1]{\left\{#1\right\}}
\newcommand{\eval}[1]{\left\langle#1\right\rangle}
\newcommand{\brr}[1]{\left[#1\right]}

\newcommand{\ds}{\displaystyle}
\renewcommand{\d}{\mathrm{d}}

\newcommand{\smalcirc}{\mbox{\,\tiny{$\circ $}\,}} %small composition circle

\DeclareMathOperator{\Ker}{Ker} % Kernel
\DeclareMathOperator{\im}{Im} % Image
\DeclareMathOperator{\ad}{ad} % adjoint
\DeclareMathOperator{\End}{End} % Endomorphisms
\DeclareMathOperator{\Hom}{Hom} % Homomorphisms
\DeclareMathOperator{\Der}{Der} % Derivations
\DeclareMathOperator{\Coder}{Coder} % Coderivation
\newcommand {\comm}[1]{{\marginpar{*}\scriptsize{\ #1 \ }}}
\begin{document}

\title{Embedding tensors on Lie $\infty$-algebras with respect to Lie $\infty$-actions}

\author{Raquel Caseiro}
\address{CMUC\\  Department of Mathematics\\University of Coimbra\\ 
3000-143 Coimbra\\ Portugal
}
\email{raquel@mat.uc.pt}

\author{Joana Nunes da Costa}
\address{CMUC\\  Department of Mathematics\\University of Coimbra\\ 
3000-143 Coimbra\\ Portugal
}
\email{jmcosta@mat.uc.pt}

\thanks{The authors are partially supported by the Center for Mathematics of the University of Coimbra - UIDB/00324/2020, funded by the Portuguese Government through FCT/MCTES}

\begin{abstract} Given two Lie $\infty$-algebras $E$ and $V$, any 
 Lie $\infty$-action of $E$ on $V$ defines a  Lie $\infty$-algebra structure on $E\oplus{}V$.   
 Some compatibility between the action and the Lie $\infty$-structure on $V$ is needed to obtain a particular  Loday $\infty$-algebra, the non-abelian hemisemidirect product.
These are the coherent actions. For coherent actions it is possible to define non-abelian homotopy embedding tensors as Maurer-Cartan elements of a convenient Lie $\infty$-algebra. Generalizing the classical case, we see that a non-abelian homotopy embedding tensor defines a Loday $\infty$-structure on $V$ and is  a morphism between  this new Loday $\infty$-algebra and $E$.\end{abstract}

\keywords{Lie $\infty$-algebra, embedding tensor, hemisemidirect product}

%\subjclass{17B10, 17B40, 17B70, 55P43} 

%%% ----------------------------------------------------------------------
\maketitle
%%% ----------------------------------------------------------------------

%%%%%%%%%%%%%%%%%%%%%%%%%%%%%%%%%%%%
%%%%%%%%%%%%%%%%%%%%%%%%%%%%%%%%%%%%
%%%%%%%%%%%%%%%%%%%%%%%%%%%%%%%%%%%%
\section*{Introduction} %
\label{sec:introduction} 

Embedding tensors  have been widely used and studied in the context of supergravity theory and (higher) gauge theories. 
The link between embedding tensors, Leibniz algebras and Lie $\infty$-algebras is exploited  in \cite{KS2020} and provides an interesting mathematical perspective that has brought this subject back to the forefront of mathematical/physical interests. Several works have appeared that provide  new and different enlightenment on the mathematical framework behind embedding tensors and their tensor hierarchies (see for instance \cite{BH2020, LSS2014, GHP2014, P2014, L2019, LP2020, LS2023} and all the references therein). From a purely algebraic  point of view,  embedding tensors are also known as  averaging operators on Lie/Leibniz algebras (see for instance \cite{A2000, R95}).  Recently,  the homotopy version of embedding tensors, their cohomology  and deformations was  studied in \cite{STZ21}. More recently, in \cite{TS2023} we can find the case where the representation space is also Lie algebra and, in this case, the embedding tensors are   called   non-abelian.

Leibniz algebras are behind  the classical (mathematical) notion of  embedding tensors. A (left) Leibniz algebra is a vector space $V$ together with a bilinear product $\smalcirc$ satisfying the Jacobi identity:
$$(x\smalcirc y) \smalcirc z-x\smalcirc (y \smalcirc z)+y\smalcirc (x \smalcirc z)=0, \quad x,y,z\in V.$$
Each Leibniz algebra  has  an associated Lie algebra called its gauge algebra $V^{Lie}=V/ \eval{x^2,\; x\in V}$ such that the projection $p:V\to V^{Lie}$ preserves the products:
$$p(x\smalcirc y)=\brr{p(x),p(y)}.$$
Left multiplication  turns $V$ into a left $V^{Lie}$-module: 
$$\Phi_{p(x)}y:=p(x)\cdot y=x\smalcirc y, \quad x,y\in V,$$
and $p:V\to V^{Lie}$,  called the \emph{embedding tensor} of $V$, is  an equivariant linear map
$$p\smalcirc \Phi_{p(x)}=\ad_{p(x)}\smalcirc p, \quad x\in V.$$

This setting can be generalized to Lie-Leibniz triples \cite{L2019, LP2020, LS2023}. A Lie-Leibniz triple is a triple $(\g,V,T)$ where  $\g$ is a Lie algebra, $V$ is a $\g$-module (given by a representation $\rho:\g\to \End(V)$) and $T:V\to \g$ is a linear map such that 
$$\brr{T(x),T(y)}_{\g}=T(\rho_{T(x)} y), \quad x,y\in V.$$ It turns out that $V$ acquires a Leibniz product given by $x\smalcirc y:=\rho_{T(x)} y$.
Therefore, in a Lie-Leibniz triple  $(\g,V,T)$,  the embedding tensor $T$ is equivariant (in the image of $T$):
$$T\smalcirc \rho_{T(x)}=\ad_{T(x)}\smalcirc T, \quad x\in V.$$
Also, it defines a Leibniz product on  $V\oplus \g$   called the hemisemidirect product. 
Lie $\infty$-algebras appeared when authors tried to explain the tensor hierarchies associated with embedding tensors.  Kotov and Strobl showed in \cite{KS2020}  that any Leibniz algebra $V$ gives rise to a universal Lie $\infty$-algebra  structure in the  minimal Lie subalgebra of $T(V)_{Lie}=(T(V), \brr{\cdot , \cdot }_c)$. In another direction, Lavau and his co-authors \cite{L2019}, \cite{LP2020},  \cite{LS2023}  gave different approaches to the  tensor hierarchies associated with embedding tensors. The construction given in \cite{KS2020} and the constructions given by Lavau and co-authors    are different and  natural relations between them  are still unclear.

More recently,  using the $V$-data  approach by Voronov \cite{V05}, Sheng et al. \cite{STZ21} %showed that embedding tensors are the Maurer-Cartan elements of a specific Lie $\infty$-algebra  and, consequentely, the deformations of embedding tensors are controlled by a deformed Lie $\infty$-algebra.
%Moreover, in the same paper 
introduced the homotopy version of embedding tensors with respect to a Lie $\infty$-representation.
Any Lie $\infty$-representation $\Phi:E\to \End(V)$ defines a Lie $\infty$-algebra $E\ltimes V$ on $E\oplus V$ but also a Leibniz $\infty$-algebra $E\ltimes^\Phi V$ in the same space: the   hemisemidirect product of $E$ by $V$. 
In  \cite{STZ21} homotopy embedding tensors  are seen as Maurer-Cartan elements of a particular  Lie $\infty$-algebra and  the  authors define  a Leibniz $\infty$-algebra  structure on  $V$.
 Moreover, they  also  establish  the existence  of a functor from Leibniz $\infty$-algebras category to the Lie $\infty$-algebras category that, in particular, induces a functor from the category of homotopy embedding tensors to that of Lie $\infty$-algebras, thus  generalizing of  the result of  Kotov and Strobl \cite{KS2020}.

Recently,  Tang and Sheng \cite{TS2023} considered  the non-abelian case.
 Given an action $\rho:\g\to \Der(\h)$ of a  Lie algebra $\g$ on another Lie algebra $\h$ the product
 $(x+v)\smalcirc (y+w) =\rho_x w +\brr{v,w}_{\h}$ 
 may not be a Leibniz product.
 For this to happen there has to be a compatibility between the action $\rho$ and the Lie algebra $\h$. These are the coherent actions. 
A non-abelian  embedding tensor (with respect to a coherent action $\rho$) is a linear map $T:\h\to \g$ such that
$$\brr{Tu,Tv}_{\g}=T\left(\rho_{Tu}v+\brr{u,v}_{\h}\right), \quad u,v\in\h.$$
When $\h$ is abelian,  usual embedding tensors are recovered.
Any non-abelian embedding tensor defines a (new) Leibniz algebra structure on $\h$ and is a Maurer-Cartan element of a differential graded Lie algebra.

In this paper we intend to  give a homotopy version of non-abelian embedding tensors, thus generalizing both definitions in \cite{STZ21} and in \cite{TS2023}.

Given two Lie $\infty$-algebras $(E,\set{l_k}_{k\in\Nn})$  and $(V,\set{m_k}_{k\in\Nn})$, a Lie $\infty$-action is a Lie $\infty$-morphism 
$$\Phi: E\to \Coder(\bar S(V))[1]$$ and  it defines a Lie $\infty$-algebra structure on $E\oplus V$, the direct product $E\ltimes V$ \cite{MZ, CC2022}.
We define coherent Lie $\infty$-actions and prove that being coherent is a necessary and sufficient condition  for $E\oplus V$ to have  a particular Loday $\infty$-structure: the non-abelian hemisemidirect product.

\begin{thmx}
Let $\Phi: (E,\set{ l_k}_{k\in\Nn}) \to (\Coder (\bar S(V))[1], \partial_{M_V}, \brr{\cdot,\cdot})$ be a Lie $\infty$-action. Consider $E\oplus V$ equipped with brackets $\set{\mathfrak{l}_n}_{n\in \Nn}$ defined by:
\begin{eqnarray*}
\lefteqn{ \mathfrak{l}_n\left(x_1+v_1,\ldots,x_{n-1}+v_{n-1}, x_n+v_n\right)=  l_n(x_1,\dots, x_n)} \nonumber\\
&&\quad +\sum_{i=1}^{n-1}\Phi^{i,n-i}(x_1,\ldots, x_{i}; v_{i+1},\dots , v_n)  
 + m_n(v_1,\dots, v_n). 
\end{eqnarray*}
Then, $E\oplus V$ is a Loday $\infty$-algebra if and only if $\Phi$ is a coherent action.
\end{thmx}

For
coherent actions we define non-abelian homotopy embedding tensors. These are (Zinbiel) comorphisms $T:T(V)\to T(E)$ that are related  to Maurer-Cartan elements of a  particular Lie $\infty$-algebra. We also  give the explicit equations that define non-abelian homotopy embedding tensors.  Then we prove that each non-abelian embedding tensor induces a (descendent) Leibniz $\infty$-algebra structure on $V$ and, finally,  we analyse some examples as well as the Lie $\infty$-algebra that controls the deformations of these tensors.

The paper is divided into three sections.
In the first section we  review Lie $\infty$-algebras and Loday $\infty$-algebras. First,  we introduce notations and conventions on graded vector spaces. Then, we define Lie and Lie$[1]$ $\infty$-algebras and its properties as well as  Loday and Loday$[1]$ $\infty$-algebras. 

The reader should be warned that after Section \ref{sec:1} we will only work with Lie$[1]$ $\infty$-algebras
 and Loday$[1]$ $\infty$-algebra and, for the sake of simplification, we discard the symbol $[1]$.

In Section \ref{sec:2}, we define coherent Lie $\infty$-actions and  the non-abelian hemisemidirect product $E\ltimes^{{\Phi}}V$. We  show that Lie $\infty$-representations on chain complexes are always coherent Lie $\infty$-actions and we analyse some examples. 

In Section \ref{sec:3}, we introduce non-abelian homotopy embedding tensors. They  are related to Maurer-Cartan elements of a  particular Lie $\infty$-algebra and we establish   explicit conditions that reveals what this Maurer-Cartan condition really means. We see how a non-abelian embedding tensor induces a  new (descendent) Loday $\infty$-algebra and we prove that non-abelian embedding tensors behave well under the functor $\mbox{(Loday $\infty$)} \to \mbox{(Lie $\infty$)}$ given in \cite{KS2020, STZ21}. 
Since the adjoint representation of a Lie $\infty$-algebra  defines a coherent Lie $\infty$-action, this special case is treated more closely. We finish by looking at  deformations of embedding tensors  and the Lie $\infty$-algebra that controls them.

\section{Lie and Loday \texorpdfstring{$\infty$}{TEXT}-algebras} \label{sec:1}%

\subsection{Notations and conventions on graded vector spaces}
We will work with finite dimensional $\Zz$-graded vector spaces and over a field $\Kk=\Rr$ or $\Kk=\Cc$.
Let $V=\oplus_{i\in\Zz} V_{i}$ be a graded vector space. Each $V_{i}$ is the homogeneous component of $V$ of degree $i$. An element $x$ of $V_{i}$ is called homogeneous with degree $|x|$.% and the sign $(-1)^{|x|}$ is often simply denoted by $(-1)^{x}$. 

The {\emph{suspension}} of $V$ is the graded vector space $sV$ defined by $(sV)_{i}=V_{i-1}$, $i\in \Zz$. The \emph{desuspension} of $V$ is the graded vector space $s^{-1}V$ defined by $(s^{-1}V)_{i}=V_{i+1}$, $i\in \Zz$

The suspension operator $s:V\to sV$ is the linear map that increases the degree of the elements by $1$:
\begin{equation*}
s:x\in V_{i} \mapsto sx:=x\in (sV)_{i+1}=V_{i}.
\end{equation*} 
The desuspension operator $s^{-1}:V\to s^{-1}V$ lowers the degree of the elements by $1$.

For each $k\in\Zz$ we will denote by $V[k]$ the graded vector space $(s^{-1})^{k}V$, that is, $(V[k])_{i}=V_{i+k}$, $i\in \Zz$. 

A \emph{morphism} $\Phi:V\to W$ between two graded vector spaces is a degree-preserving linear map, i.e., a collection of linear maps $\Phi_{i}:V_{i}\to W_{i}$, $i\in\Zz$. We call $\Phi:V\to W$ a \emph{morphism of degree $k$}, for some $k\in\Zz$, if it is a morphism between $V$ and $W[k]$. The space of morphisms of degree $k$ is denoted by $\Hom^{k}(V,W)$ and
$\Hom(V,W)=\bigoplus_{k\in\Zz}\Hom^{k}(V,W)$.

Given two graded vector spaces $V$ and $W$, their \emph{direct sum} $V\oplus W$ (resp. \emph{tensor product} $V\otimes W$) is the graded vector space with grading
\begin{equation*}
(V\oplus W)_{i}=V_{i}\oplus W_{i} \quad \mbox{(resp. $(V\otimes W)_{i}=\oplus_{k+j=i} V_{j}\otimes W_{k}$).}
\end{equation*}

For each $k\in \Nn_{0}$, let $T^{k} (V)={\otimes^{k}}V$. The tensor algebra over $V$ is 
$T(V)=\oplus_{k\geq 0}T^{k} (V)$ and the reduced tensor algebra over $V$ is $\bar T(V)=\oplus_{k\geq 1} T^{k} (V)$.
%For the tensor algebra of a direct sum $V\oplus W$ we  consider the identification
%\begin{equation*}
%    T^n(V\oplus W)\simeq\!\!\!\!\!\bigoplus_{\scriptsize{\begin{array}{cc}
%        k\geq 1\\ n=i_1+\ldots +i_{2k}\\ i_1,\ldots , i_{2k}\geq 0
%    \end{array}}} \!\!\!\!\!T^{i_1}(V)\otimes T^{i_2}(W) \otimes \ldots T^{i_{2k-1}}(V)\otimes T^{i_{2k}}(W).
%\end{equation*}

We follow the standard Koszul sign convention: for homogeneous morphisms $f:V\to W$ and $g:E\to L$, of degrees $|f|$ and $|g|$, respectively, the tensor product 
$f\otimes g:V\otimes E\to W\otimes L$ is the morphism of degree $|f|+|g|$ given by 
\begin{equation*}
(f\otimes g)(x\otimes y)=(-1)^{|g|\,|x|}f(x)\otimes g(y),\quad x\in V,\, y\in E.
\end{equation*}

For each $k\geq 1$, 
the permutation group of order $k$, $S_{k}$, acts on $T^{k} (V)$ by
$$
 \sigma(v_{1}\otimes\ldots\otimes v_{k})=\epsilon(\sigma; v_{1}\otimes\ldots\otimes v_{k}) v_{\sigma(1)}\otimes \ldots \otimes v_{\sigma(n)}, $$
for $\sigma\in S_{k}$ and $v_{1},\ldots, v_{k}\in V$, where $\epsilon(\sigma; v_{1}\otimes\ldots\otimes v_{k})$ stands for \emph{Koszul sign}. Although the Koszul sign depends both on $\sigma$ and $v_{1}\otimes \ldots \otimes v_{k}$, for sake of simplification, we will simply denote it by
$\epsilon(\sigma)$. 
A permutation $\sigma$ of order $n=i_{1}+\ldots+i_{k}$ is called a \emph{$(i_{1},\ldots, i_{k})$-unshuffle} if 
\begin{eqnarray*}
\sigma(1)<&\ldots& < \sigma(i_{1}),\\
\sigma(i_{1}+1)<&\ldots &<\sigma(i_{1}+i_{2}), \\
&\ldots&\\
\sigma{(i_{1}+\ldots +i_{k-1}+1)}< &\ldots& < \sigma(n).
\end{eqnarray*}
The set of $(i_{1},\ldots, i_{k})$-unshuffles is denoted by $Sh(i_{1},\ldots,i_{k})$.

An important subset of unshuffles will be considered. An \emph{increasing $(i_{1},\ldots, i_{k})$-unshuffle} $\sigma$ is an element of $Sh(i_{1},\ldots,i_{k})$ such that
$$
\sigma(i_{1})<\sigma(i_{1}+i_{2})<\ldots < \sigma(n).
$$
The set of increasing $(i_{1},\ldots, i_{k})$-unshuffles is denoted by $\widetilde{Sh}(i_{1},\ldots,i_{k})$.

%In particular,
%a permutation $\sigma$ of order $k$ is called an \emph{$(i,k-i)$-unshuffle} if 
%\begin{equation*}
%\sigma(1)<\ldots <\sigma(i) \mbox{ and } \sigma(i+1)<\ldots <\sigma(k).
%\end{equation*}
%The set of $(i,k-i)$-unshuffles is denoted by $Sh(i,k-i)$

An element of $T^{k} (V)$ is called a \emph{symmetric tensor} (of order $k$) if it is invariant by the action of $S_{k}$. The set $T(V)^S$ of all symmetric tensors  is a graded subalgebra of $T(V)$, $i:T(V)^{S}\hookrightarrow T(V)$.
The symmetrization map $\pi:T(V)\to T(V)^{S}$,
$$
\pi(v_{1}\otimes\ldots,\otimes v_{n})=\sum_{\sigma\in S_{k}}\frac{\epsilon(\sigma)}{k!}v_{\sigma(1)} \otimes\ldots \otimes v_{\sigma(k)}, \quad \quad v_{1},\ldots, v_{k}\in V,
$$
allows the identification of  
the set of all symmetric tensors with the quotient 
$$S(V)=T(V)/\eval{x\otimes y-(-1)^{|x||y|}y\otimes x}.
$$ 
Therefore, the quotient $S(V)$ is a graded commutative algebra whose product we denote by $\cdot$:
$$
v_{1}\cdot\ldots\cdot v_{n}\simeq\pi(v_{1}\otimes \ldots \otimes v_{n}), \quad v_{1},\ldots,v_{n}\in V.
$$
%We will also denoted by $\ds N:S(V)\to T(V)$ the isomorphism induced by the symmetrization map:
%The isomorphism is given by %the injective map
%\begin{equation}\label{eq:def:comorphism:isomorphism}
%[v_{1}\otimes\ldots \otimes v_{k}]\mapsto \sum_{\sigma\in S_{k}}\frac{\epsilon(\sigma)}{k!}v_{\sigma(1)} \otimes\ldots \otimes v_{\sigma(k)}, \quad \quad v_{1},\ldots, v_{k}\in V.
%\end{equation}

%$$
%F:S(V)\to W \leftrightarrow F^{S}=F\smalcirc N:T(V)\to W.
%$$

A linear map $F:T(V)\to W$ is {\emph{symmetric}} if $F\smalcirc \pi=F$. Symmetric linear maps are in one-to-one correspondence with linear maps from $S(V)$ to $W$ and we use this correspondence to write
$$F(v_{1},\ldots, v_{k})=F(v_{1}\cdot\ldots\cdot v_{k}), \quad v_{1},\ldots, v_{k}\in V.$$ 

Given a linear map $F: T(V) \to W$, its {\emph{symmetrization}} is the morphism $F^{S}:S(V)\to W$, defined by 
$$F^{S}(v_{1},\ldots, v_{k})=\sum_{\sigma\in S_{k}}\frac{\epsilon(\sigma)}{k!}F(v_{\sigma(1)} \otimes\ldots \otimes v_{\sigma(k)}), \quad v_{1},\ldots, v_{k}\in V.
$$

%We use this isomorphism to identify $\Hom(\bar S(V), W)$ with the set of {\emph symmetric morphisms} between $T(V)$ and $W$.

% The projection $\pi:T(V)\to S(V)$ is an algebra morphism 
% $$
% \pi(v_{1}\otimes\ldots \otimes v_{k})=v_{1}\cdot\ldots\cdot v_{n}, \quad v_{1},\ldots,v_{n}\in V,
% $$
% $\ds \pi\smalcirc N=\mathrm{Id}_{S(V)}$ and $N\smalcirc \pi_{|_{N(S(V))}}=\mathrm{Id}_{N(S(V))}.$
%

%Consider
%$$
%\sigma_{1}:=(\sigma(1),\ldots ,\sigma(p)), \mbox{ and } \sigma_{2}:= (\sigma(p+1),\ldots ,\sigma(k)).
%$$
%We will say that $\sigma=\sigma_{1}\cup \sigma_{2}$ and $\ds \sigma_{1}<\sigma_{2}$ if $\sigma(p)<\sigma(k)$.

The  reduced tensor algebra $\bar T(V)$ can be equipped with the coshuffle coproduct 
%% COMENT %%
%\comm{Não preciso dos dois coprodutos, acho...
%two coassociative coalgebra products: 
%
%1) The {\emph{deconcatenation coproduct}} $\Delta:\bar T(V)\to \bar T(V)\otimes \bar T(V)$ is given by
% \begin{equation*}
%\Delta(v)=0, \quad v\in V,
%\end{equation*}
%\begin{equation*}
%\Delta(v_{1}\otimes \ldots \otimes v_{k})= \sum_{p=1}^{k} \left(v_{1} \otimes\ldots \otimes v_{p}\right) \otimes \left(v_{p+1} \otimes \ldots \otimes v_{k}\right).
%\end{equation*}
%The pair $(\bar T(V),\Delta)$ is simply called the \emph{reduced tensor coalgebra}. %and will be denoted simply by $\bar T(V)=(\bar T(V),\Delta)$.
%
%2)The {\emph{coshuffle coproduct}} }
%%%%%%%%%
$\Delta^{c}:\bar T(V)\to \bar T(V)\otimes  \bar T(V)$  defined by
\begin{equation*}
\Delta^{c}(v)=0, \quad v\in V,
\end{equation*}
\begin{equation*}
\Delta^{c}(v_{1}\otimes \ldots \otimes v_{k})=\!\!\!\!\!\!\!\!\!\! \sum_{\substack{p=1\\ \sigma\in Sh(p,k-p)}}^{k-1} \!\!\!\!\!\!\!\!\!\! \epsilon(\sigma)\left(v_{\sigma(1)} \otimes\ldots \otimes v_{\sigma(p)}\right) \otimes \left(v_{\sigma(p+1)} \otimes \ldots \otimes v_{\sigma(k)}\right),
\end{equation*}
for all $v_{1},\ldots,v_{k}\in V$.
The coalgebra $(\bar T(V),\Delta^{c})$ is non-unital, cocommutative and coassociative. It is called the (reduced)\emph{coshuffle coalgebra} and is denoted by $\bar T^{c}(V)$. 
%In fact $(T(V), \Delta^c, \otimes)$ is a  conilpotent bialgebra called \emph{coshuffle bialgebra}.

The coshuffle coproduct $\Delta^{c}$ induces a cocommutative coassociative coproduct in $\bar S(V)$, which will also be denoted by $\Delta^{c}$: 

\begin{equation*}
\Delta^{c}(v)=0, \quad v\in V,
\end{equation*}
\begin{equation*}
\Delta^{c}(v_{1}\cdot \ldots \cdot v_{k})=\!\!\!\!\!\!\!\!\!\! \sum_{\substack{p=1\\ \sigma\in Sh(p,k-p)}}^{k-1} \!\!\!\!\!\!\!\!\!\! \epsilon(\sigma)\left(v_{\sigma(1)} \cdot\ldots \cdot v_{\sigma(p)}\right) \otimes \left(v_{\sigma(p+1)} \cdot\ldots \cdot v_{\sigma(k)}\right),
\end{equation*}
for all $v_1, \ldots,v_k\in V$.
Its cogenerator is the projection map $p:\bar S(V)\to V$.

Let 
$\mathrm{Prim}(V)=\set{x\in \bar T(V):\; \Delta^c(x)=0}
$ be the set of primitive elements of $T^c(V)$. It is a graded Lie subalgebra of $\mathrm{Lie}(V)=(T(V),\brr{\cdot , \cdot }_c)$, where $\brr{\cdot , \cdot }_c$ denotes the commutator of the graded algebra $(T(V),\otimes)$. 
In fact, $\mathrm{Prim}(V)$ is the free Lie algebra generated by $V$ \cite{R03}.
Moreover, $(\bar T(V),\Delta^c)$ and $(\bar S(\mathrm{Prim}(V)), \Delta^c)$
are isomorphic coalgebras \cite{STZ21, KS2020}.

%%COMMENT %%
%%% NOTACAO SWEEDLER MAS ACHO QUE NAO PRECISO %%%
%\comm{
%\begin{rem}
%We will use Sweedler notation for the coshuffle coproduct:
%$$\Delta^{c}(v)=v_{(1)}\otimes v_{(2)}, \quad v\in \bar T(V) .$$
%\end{rem}
%}

%%% NOTACAO SWEEDLER MAS ACHO QUE NAO PRECISO %%%
%\begin{rem}
%The symmetrization map $\pi:\bar T^{c}(V)\to (\bar S(V),\Delta^{c})$ %and the injection $N:(\bar S(V),\Delta^{c})\to \bar T^{c}(V)$ defined in (\ref{eq:def:comorphism:isomorphism}) are comorphisms. 
%is a comorphism. We will
% also use Sweedler notation for this coproduct:
%$$\Delta^{c}(v)=v_{(1)}\otimes v_{(2)}, \quad v\in \bar S(V).$$
%\end{rem}

%%%%%%%%%%%%%%%%%%%%%%%%%%%%%%%%%%%%%%
\subsection{Lie\texorpdfstring{$[1]$ $\infty$}{TEXT}-algebra}
%%%%%%%%%%%%%%%%%%%%%%%%%%%%%%%%%%%%%%

A linear map $Q:\bar S(V)\to \bar S(V)$ of degree $|Q|$ is called a \emph{coderivation} of  $(\bar S(V),\Delta^{c})$ (of degree $|Q|$), if 
% A {\emph{coderivation}} of $(\bar S(V),\Delta^{c})$ is a linear map $Q:\bar S(V)\to \bar S(V)$ such that
$$
\Delta^{c} Q=(Q\otimes \mathrm{Id} + \mathrm{Id}\otimes Q)\Delta^{c}.
$$

Let $\Coder(\bar S(V))$ be the set of coderivations of the reduced symmetric coalgebra $(\bar S(V),\Delta^{c})$.
It is well known that $\Coder(\bar S(V))$ together with the graded commutator becomes a graded Lie algebra:
$$
\brr{Q,P}_{c}=Q\smalcirc P - (-1)^{|Q||P|}P\smalcirc Q,\quad Q,P \in\Coder(\bar S(V)). 
$$

\begin{defn}
A \textbf{Lie$[1]$ $\infty$-algebra} (or a symmetric {Lie $\infty$-algebra}) is a graded vector space $V$ together with a degree $+1$ coderivation $M_V$ of the reduced symmetric coalgebra $(\bar S(V),\Delta^{c})$ squaring zero.
\end{defn}

By a standard argument, any coderivation $Q\in \Coder(\bar S(V))$ defines (and is defined by) its restriction maps $Q_{k}:S^{k}(V)\to V$, $k\geq 1$.
This way we have an alternative definition of Lie$[1]$ $\infty$-algebra.

\begin{defn}
A \textbf{Lie$[1]$ $\infty$-algebra} is a graded vector space $V$ together with a family a degree $+1$ linear maps $\set{l_{k}:S^{k}(V)\to V}_{k\geq 1}$ such that
\begin{equation}\label{eq:def:Lie:infty:algebra}
\sum_{\substack{i= 1\\ \sigma\in Sh(i,n-i)}}^{n} \!\!\!\!\! \epsilon(\sigma)
l_{{n-i+1}}\left(l_{i}(v_{\sigma(1)} ,\ldots, v_{\sigma(i)}), v_{\sigma(i+1)},\ldots, v_{\sigma(n)} \right)=0,
\end{equation} 
for all $n\in \Nn$, and all homogeneous elements $v_{1},\dots,v_{n}\in V$.
\end{defn}

\begin{rem}
The isomorphisms $\otimes^{k} s^{-1}:T^{k} (V)\to T^{k} (s^{-1}V)$, $k\in\Nn$, yield an isomorphism between $\Hom(\bar T(s^{-1}V), s^{-1}V) $ and $ \Hom(\bar T(V), V)$, given by:
\begin{equation}\label{eq:isomorphism:infty:infty1}
Q_{k}\in\Hom^{|Q_k|}(T^{k} (s^{-1}V), s^{-1}V) \mapsto q_{k}=s\smalcirc Q_{k}\smalcirc \otimes^{k}s^{-1}\in \Hom^{|Q_k|+1-k} (T^{k} (V), V).
\end{equation}
The inverse morphism is given by
$\ds
Q_{k}=(-1)^{\frac{k(k-1)}{2}} s^{-1}\smalcirc q_{k}\smalcirc \otimes^{k} s.
$

This isomorphism also provides an isomorphism between the symmetric algebra $S(s^{-1}V)$ and the exterior algebra $\wedge V$.
 Using this isomorphism one has the definition of \emph{Lie $\infty$-algebra}: a vector space $V$ equipped with a family of skew-symmetric brackets $l'_{k}:\wedge^{k}V\to V$, $k\geq 1$, of degree $2-k$, satisfying a collection of equations equivalent to (\ref{eq:def:Lie:infty:algebra}).
 
\end{rem}

%%%%

\begin{ex}%[Lie$[1]$ algebra]
A graded Lie$[1]$-algebra is a Lie$[1]$ $\infty$-algebra $V=\oplus_{i\in\Zz}V_i$ such that $l_n = 0$, for $n \neq 2$. Then the degree $0$ bilinear map on $sV=V[-1]$ is defined by:
 \begin{equation}
 \label{eq:decalage}
[sx,sy] := l'_{2}(sx,sy)=(-1)^{i} s(l_2(x,y)), \quad   x \in V_i, y\in V_j, 
 \end{equation}
is a graded Lie bracket.
In particular, if $V=V_{-1}$ is concentrated in degree $-1$, we get a Lie algebra structure.
\end{ex}

\begin{ex}%[ DG Lie$[1]$ algebra]
A differential graded Lie$[1]$-algebra (DGLA$[1]$) is a Lie$[1]$ $\infty$-algebra $V=\oplus_{i\in\Zz}V_i$ such that $l_n=0$, for $n \neq 1$ and $n \neq 2$.
% $l_2=\brr{\, -\, , \, -\, }$.
Then $\d:=l_ 1$ is a degree $+1$ linear map $\d:V\to V$ squaring zero and satisfies the following compatibility condition with the bracket $ \set{\cdot,\cdot}= l_2(\cdot,\cdot)$:
\begin{equation*}
%\label{eq:gradedLie}
\left\{\begin{array}{l}
\d\set{x,y} + \set{\d(x),y} + (-1)^{|x|}\set{x, \d(y)}=0,\\
\set{\set{x,y},z} + (-1)^{|y||z|}\set{\set{x,z},y} + (-1)^{|x|}\set{x,\set{y,z}}=0.
\end{array}\right. 
\end{equation*}
\end{ex}

\begin{ex}\label{ex:DGLA:End:E}
Let $(V=\oplus_{i\in\Zz}V_i, \d)$ be a cochain complex. Then $\End (V)[1]=(\oplus_{i\in\Zz} \End_i V)[1]$ has a natural DGLA$[1]$  structure with $l_1=\partial_{\d}, \;\;l_2=\brr{\cdot , \cdot }$ given by:
\begin{equation*}
\left\{\begin{array}{l}
 \partial_{\d}\phi%=\tilde l_1\phi
=-\d \smalcirc \phi + (-1)^{|\phi|+1} \phi\smalcirc \d,\\
 \brr{\phi,\psi}%=\tilde l_2(\phi,\psi)
=(-1)^{|\phi|+1}\left(\phi\smalcirc \psi - (-1)^{(|\phi|+1)(|\psi|+1)} \psi\smalcirc\phi \right),
\end{array}\right.
\end{equation*}
for $\phi, \psi$ homogeneous elements of $\End(V)[1]$.
\end{ex}

Let $(E,M_E\equiv\set{l_k}_{k\in\Nn})$ and $(V,M_V\equiv\set{m_k}_{k\in\Nn})$ be Lie $\infty$-algebras. A \emph{Lie ${\infty}$-morphism } $\ds \Phi:E \rightarrow V$
is given by a collection of degree zero linear maps:
$$
\Phi_k:S^k(E)\to V,\quad k\geq 1,
$$
such that, for each $n\geq 1$,
\begin{align}\label{eq:def:Lie:infty:morphism}
&\sum_{\begin{array}{c} \scriptstyle{k+l=n}\\ \scriptstyle{\sigma\in {Sh}(k,l)}\\ \scriptstyle{l\geq 0, \, k\geq 1}\end{array}}\!\!\!\!\!\!\! \varepsilon(\sigma) \Phi_{1+l}\Big(l_k(x_{\sigma(1)},\ldots, x_{\sigma(k)}), x_{\sigma(k+1)}, \ldots, x_{\sigma(n)}\Big) = \\ &=\!\!\!\!\!\!\!\!\!\!\! \sum_{\begin{array}{c}\scriptstyle{k_1+\ldots+ k_j=n} \\ \scriptstyle{\sigma\in \widetilde{Sh}(k_1,\ldots, k_j)}\end{array}} \!\!\!\!\!\!\!  {\varepsilon(\sigma)}\,m_j\Big(\Phi_{k_1}(x_{\sigma(1)}, \ldots x_{\sigma(k_1)}),
\Phi_{k_2}(x_{\sigma(k_1+1)}, \ldots x_{\sigma(k_1+k_2)}),\ldots,\nonumber \\
&\hspace{4cm}  \Phi_{k_j}(x_{\sigma(k_1+\ldots+k_{j-1}+1)}, \ldots, x_{\sigma(n)})\Big).\nonumber
\end{align}

\noindent{If $\Phi_k=0$ for $k\neq 1$, then $\Phi$ is called a \emph{{strict} Lie ${\infty}$-morphism.}}

Considering the coalgebra morphism $\Phi: \bar S(E)\to \bar S(V)$ defined by
the collection of degree zero linear maps $$
\Phi_k:S^k(E)\to V,\quad k\geq 1,
$$
  we see that Equation \eqref{eq:def:Lie:infty:morphism} is equivalent to
$\Phi$ preserving the Lie $\infty$-algebra structures:
$$\Phi \smalcirc M_E=M_V\smalcirc \Phi.$$

 Let us now give the definition of a (curved) representation of a Lie $\infty$-algebra.
 
 A complex $(V,\d)$ induces a  natural symmetric DGLA$[1]$ structure on $\End(V)[1]$  (see Example \ref{ex:DGLA:End:E}).
 A \emph{representation} of a Lie $\infty$-algebra $(E,M_E\equiv\set{l_k}_{k\in\Nn})$ on a complex $(V,\d)$ is a Lie $\infty$-morphism $$\Phi:(E,\set{ l_k}_{k\in\Nn})\rightarrow (\End(V)[1],\partial_{\d}, \brr{\cdot , \cdot}),$$
 {i.e., $\Phi \smalcirc M_E= M_{\End(V)[1]}\smalcirc \Phi$, where  $ M_{\End(V)[1]}$ is the coderivation determined by $\partial_{\d}+ \brr{\cdot , \cdot}$.}

 Equivalently, a  representation of $E$ on $(V,\mathrm{d})$ is defined by a collection of degree $+1$ maps
$$\Phi_k:S^k(E)\to \End(V),\quad k\geq 1,$$
such that, for each $n\geq 1$,
\begin{align*} %\label{eq:def:representation}
\lefteqn{\sum_{\begin{array}{c} \scriptstyle{i=1}\\ \scriptstyle{\sigma\in Sh(i,n-i)}\end{array}}^{\scriptstyle{n}}\!\!\!\!\!\!\!\!\!\! \varepsilon(\sigma)\Phi_{n-i+1}\left(l_i\left(x_{\sigma(1)}, \ldots, {x_{\sigma(i)}}\right), {x_{\sigma(i+1)}}, \ldots, {x_{\sigma(n)}}\right)} \\
&= & \!\!\!\!  \partial  \Phi_n(x_1,\ldots, x_n)\!+\!  \!\!\!\!\!\!\!\!\!\!\sum_{\begin{array}{c} \scriptstyle{j=1}\\ \scriptstyle{\sigma\in \widetilde{Sh}(j,n-j)}\end{array}}^{\scriptstyle{n-1}} \!\!\!\!\!\!\!\!\!\!\varepsilon(\sigma)\!\!\brr{\Phi_j({x_{\sigma(1)}}, \ldots, {x_{\sigma(j)}}) , \Phi_{n-j}({x_{\sigma(j+1)}}, \ldots, {x_{\sigma(n)}}) }. \nonumber
\end{align*}

%% Maurer Cartan elements %%
A \emph{Maurer-Cartan element} of a Lie$[1]$ $\infty$-algebra $(E,\set{l_k}_{k\in\Nn})$ is a degree zero element $e$ of $E$ such that
\begin{equation} \label{def:MC:element}
\sum_{k \geq 1}\frac{1}{k!}\, l_{k}(e, \ldots, e) =0.
\end{equation}
The set of Maurer-Cartan elements of $E$ is denoted by $\textrm{MC}(E)$.
Let $e$ be a Maurer-Cartan element of $(E,\set{l_k}_{k\in\Nn})$ and set, for $k\geq 1$,
\begin{equation} \label{def:twisting:MC}
l_k^{e}(x_1, \ldots, x_k):= \sum_{i\geq 0}\frac1{i!}\, l_{k+i}(e, \ldots, e, x_1, \ldots, x_k).
\end{equation}
Then, $(E,\set{l_k^{e}}_{k\in\Nn})$ is a Lie$[1]$ $\infty$-algebra, called \emph{twisting of} $E$ by $e$ \cite{G}.
 For filtered, or even weakly filtered Lie $\infty$-algebras, the convergence of the infinite sums defining Maurer-Cartan elements and twisted Lie $\infty$-algebras (Equations (\ref{def:MC:element}) and (\ref{def:twisting:MC})) is guaranteed (see \cite{G,FZ2015}).

%%COMMENT%%
%% MINIMAL MODEL OF A LIE INFTY ALGEBRA%%%%
%\comm{
%XXXXXXXXXX Have to decide if we put this or not XXXXXXXXX
%
%
%\noindent
%Recall that for $(E, \set{l_k}_{k\geq 1})$, a symmetric Lie $\infty$-algebra, equations (\ref{eq:def:symm:L:infty:algebra}) establish:
%\begin{itemize}
%	\item[(i)] for $n=1$, that
%	$l_1\smalcirc l_1=0$, so that $l_1:E_\bullet \to E_{\bullet +1}$ is a differential on $E$ and we have an associated cohomology $H^{\bullet}(E,l_1)$;
%	\item[(ii)] for $n=2$, that
%	$$l_1(l_2(x_1,x_2))+ l_2(l_1(x_1),x_2) + (-1)^{|x_1|}l_2(x_1,l_1(x_2))=0,$$
%	so that the bracket $l_2$ induces a graded symmetric bracket on $H^{\bullet}(E,l_1)$;
%	\item[(iii)] for $n=3$, that the previously defined symmetric bracket on $H^{c}(E,l_1)$ satisfies the graded symmetric Jacobi identities as in equations (\ref{eq:gradedLie}).
%\end{itemize}
%
%Hence:
%
%\begin{prop}\label{prop:cohom}Let $(E, \set{l_k}_{k\geq 1})$ be a symmetric Lie $\infty$-algebra.
%The graded vector space $H^{c}(E,l_1)$ has a natural graded symmetric Lie algebra structure.
%\end{prop}
%XXXXXXXXXXX}

\subsection{Loday \texorpdfstring{$\infty$}{TEXT}-algebra and Loday\texorpdfstring{$[1]$ $\infty$}{TEXT}-algebra}
Let us recall the definition and some properties  of Loday $\infty$-algebras and Loday[1] $\infty$-algebras. More details can be found in  \cite{AP2010, U2011}.

Let $V$ be a graded vector space. The coproduct 
$$
\Delta^{Z}:\bar T(V) \to \bar T(V)\otimes \bar T(V)
$$
given by
\begin{equation*}
\Delta^{Z}(v)=0, \quad v\in V,
\end{equation*}
\begin{equation*}
\Delta^{Z}(v_{1}\otimes \ldots \otimes v_{k})=\!\!\!\!\!\!\!\!\!\! \sum_{\substack{p=1\\ \sigma\in Sh(p,k-p-1)}}^{k-1} \!\!\!\!\!\!\!\!\!\! \epsilon(\sigma)\left(v_{\sigma(1)} \otimes\ldots\otimes v_{\sigma(p)}\right) \otimes \left(v_{\sigma(p+1)} \otimes \ldots\otimes v_{\sigma(k-1)}\otimes v_{k}\right),
\end{equation*}
for all $v_{1},\ldots,v_{k}\in V$, provides a (cofree, conilpotent) Zinbiel  coalgebra structure to $\bar T(V)$, i.e, this coproduct 
%The pair $(\bar T(V),\Delta^{Z})$ will be denoted by $T^{Z}(V)$.
satisfies
$$
(\mathrm{Id}\otimes \Delta^{Z})\Delta^{Z} = (\Delta^{Z}\otimes \mathrm{Id})\Delta^Z + (\tau\smalcirc \Delta^{Z}\otimes \mathrm{Id})\Delta^{Z},
$$
where $\tau$ is the twisting map: $\tau(v\otimes w)=(-1)^{|v||w|}w\otimes v$. 

\begin{rem} Notice that $\Delta^{c}=\Delta^{Z} + \tau\smalcirc \Delta^{Z}$, this means that the coshuffle coproduct is the symmetrization of the Zinbiel coproduct.
\end{rem}

The pair $(\bar T(V),\Delta^{Z})$ is the \emph{Zinbiel reduced tensor coalgebra} and we will denote it by $T^{Z}(V)$. The cogenerator of $T^Z(V)$ is the projection map $p:\bar T(V)\to V$.

%A coderivation in this coalgebra is a linear map $Q:\bar T(V)\to \bar T(V)$, satisfying
%\begin{equation*}
%Q\Delta=(Q\otimes 1 + 1\otimes Q) \Delta. 
%\end{equation*}
%The set of coderivations of $\bar T(V)$ is denoted by $\Coder(\bar T(V))$.

Let $\Coder(T^{Z}(V))$ be the set of coderivations of $T^{Z}(V)=(\bar T(V),\Delta^{Z})$.
There is a (standard) isomorphism between $\Coder(T^Z(V))$ and $\Hom(\bar T(V), V)$ that we describe now:
each linear map $Q_{k}:T^{k} (V)\to V$ ($k\geq 1$) of degree $|Q|$ defines the coderivation $Q$ given by
\begin{eqnarray*}%\label{eq:def:coderivation:Zinbiel:coalgebra}
Q(v_{1}\otimes \ldots\otimes v_{n})&=&\sum_{k=1}^{n}\!\!\!\!\!\!\!\! \sum_{\substack{i= 0, \\ \sigma\in Sh(i,k-1)}}^{k-1} \!\!\!\!\! (-1)^{|Q|(|v_{\sigma(1)}| +\ldots + |v_{\sigma(i)}|)}\epsilon(\sigma)
v_{\sigma(1)} \otimes\ldots\otimes v_{\sigma(i)}\otimes \nonumber\\
&& \otimes Q_{k}(v_{\sigma(i+1)},\ldots, v_{\sigma(k-1)}, v_{k} )\otimes v_{k+1}\otimes \ldots \otimes v_{n},
\end{eqnarray*}
for $v_{1},\ldots,v_{n}\in V$, $n\in \Nn$. 

The inverse mapping associates to each coderivation $Q:\bar T^Z(V)\to \bar T^Z(V)$ its restriction:
 $$p\smalcirc Q:\bar T(V)\to V,$$ where $p:\bar T(V)\to V$ stands for the projection onto $V$.

The (graded-)commutator gives to $\Coder(T^Z(V))$ a structure of a graded Lie algebra:
\begin{equation*}
\brr{Q,P}_{c}=Q\smalcirc P-(-1)^{|Q||P|} P\smalcirc Q.
\end{equation*}

By the isomorphism defined above, the graded vector space $\Hom(\bar T(V),V)$ acquires a graded Lie algebra bracket $\brr{\cdot , \cdot }_{B}$, known as the \emph{Balavoine bracket} \cite{B97}. 

\begin{rem}
Any coderivation $Q\equiv\set{Q_{k}}_{k\geq 1}$ of $T^{Z}(V)$ is also a coderivation of the coshuffle coalgebra $T^{c}(V)$. 
Since $\bar T^c(V)$ and $\bar S(\mathrm{Prim (V)})$ are isomorphic coalgebras, $Q$ defines a  coderivation $Q^{S}$ of $\bar S(\mathrm{Prim (V)})$.

\end{rem}
\begin{rem} 
Each coderivation $Q\equiv\set{Q_{k}}_{k\geq 1}$ of $(\bar S(V),\Delta^{c})$ defines a coderivation of $T^{Z}(V)$, which we will denote by $Q^{Z}$ (or simply by $Q$ when there is no danger of confusion). It is defined by the same family of (symmetric) linear maps $\set{Q_{k}:\bar S(V)\to V}_{k\geq 1}$ and these coderivations are related by
$$
\pi\smalcirc Q^{Z}= Q\smalcirc \pi.
$$
\end{rem}

\begin{defn}
A \textbf{Loday ${\infty}$-algebra} $(V,Q)$ is graded vector space $V$ together with a coderivation $Q$ of the Zinbiel coalgebra $\bar T^{Z}(s^{-1}V)$, of degree $+1$, that squares zero. 
\end{defn}

A Loday $\infty$-algebra $(V,Q)$ may be equivalently defined by a family of degree $+1$ linear maps $Q_{k}:T^{k} (s^{-1}V)\to s^{-1}V$ satisfying: for each $n\in\Nn$,
\begin{eqnarray}
0&=&\sum_{k=1}^{n} \sum_{\substack{i= 0\\ \sigma\in Sh(i,k-1)}}^{n-k} \!\!\!\!\! (-1)^{|v_{\sigma(1)}| +\ldots + |v_{\sigma(i)}|}\epsilon(\sigma)
Q_{{n-k+1}}\left(v_{\sigma(1)} ,\ldots, v_{\sigma(i)}, \right. \label{eq:def:Loday:infty:algebra} \\ \nonumber
&&\quad \left.Q_{k}(v_{\sigma(i+1)},\ldots, v_{\sigma(k+i-1)}, v_{i+k} ), v_{i+k+1}, \ldots , v_{n}\right), \end{eqnarray} 
for all $v_{1},\ldots, v_{n}\in s^{-1}V$.

\begin{rem}
The isomorphism $\Hom(\bar T(s^{-1}V), s^{-1}V) \simeq \Hom(\bar T(V), V)$ given by (\ref{eq:isomorphism:infty:infty1})
%\begin{equation*}\label{eq:isomorphism:infty:infty1}
%Q_{k}\in\Hom^{|Q|}(T^{k} (s^{-1}V), s^{-1}V) \mapsto q_{k}=s\smalcirc Q_{k}\smalcirc \otimes^{k}s^{-1}\in \Hom^{|Q|+1-k} (T^{k} V, V).
%\end{equation*}
yields an equivalent definition of a Loday $\infty$-algebra in terms of maps $q_{k}\in \Hom^{2-k}(T^{k} (V),V)$, $k\in \Nn$.
So, a Loday $\infty$-algebra may be defined by a family of brackets $q_{k}\in \Hom^{2-k}(T^{k} (V),V)$ satisfying a collection of equations equivalent 
to (\ref{eq:def:Loday:infty:algebra}).
\end{rem}

A codifferential (i.e. a coderivation of degree $+1$ squaring zero) of $T^{Z}(V)=(\bar T(V),\Delta^{Z})$ induces a Loday $\infty$-algebra structure on $sV$. This justifies the next definition:

\begin{defn}
A \textbf{Loday$[1]$ $\infty$-algebra} is a graded vector space $V$ together with a codifferential of the Zinbiel coalgebra $T^{Z}(V)=(\bar T(V),\Delta^{Z})$. 
\end{defn}

\begin{ex}
A {\emph{graded Loday algebra}} is a Loday $\infty$-algebra $V$ such that 
$q_{n}=0$, for $n\neq 2$.
Then the bracket $\brr{\cdot , \cdot }:=q_{2}:V\otimes V\to V$ satisfies the graded Jacobi identity
$$
\brr{x,\brr{y,z}} =\brr{\brr{x,y},z}+(-1)^{|x||y|}\brr{y,\brr{x,z}}, \quad x,y,z\in V.
$$ 

In this case,
 $s^{{-1}}V$ is a Loday$[1]$ $\infty$-algebra 
 with codifferential defined by the linear map $Q_{2}(v,w):=(-1)^{|v|-1}s^{-1}\brr{sv,sw}$, $v,w\in s^{-1}V$, that satisfies 
 $$
 Q_{2}(Q_{2}(v,u),w)+(-1)^{|v|}Q_{2}(v,Q_{2}(u,w))+(-1)^{(|v|+1)|u|}Q_{2}(u,Q_{2}(v,w))=0,
$$
for all $v,u,w\in s^{-1}V$.
\end{ex}

\begin{ex}
A {\emph{graded differential Loday algebra}} is Loday $\infty$-algebra $V$ such that
$q_{n}=0$ for $n\neq 1$ and $n\neq 2$. Then
$\d:=q_{1} $ is a degree $+1$ linear map squaring zero and $\brr{\, ,\,}:=q_{2}$ a Loday bracket satisfying the compatibility condition:
\begin{equation*}
\d \brr{x,y}=\brr{\d x,y} + (-1)^{|x|}\brr{x,\d y}, \quad x,y\in V.
\end{equation*}

The graded vector space $s^{-1}V$ is a Loday$[1]$ $\infty$-algebra with the codifferential 
$Q$ of $\bar T(s^{-1}V)$ given by the maps
 $$Q_{1}(v)=s^{-1}\d sv, \quad v\in s^{-1}V,$$ and 
 $$Q_{2}(v,w)=(-1)^{|v|-1}s^{-1}\brr{sv , sw }, \quad v,w\in s^{-1}V.$$

They satisfy the compatibility conditions
\begin{eqnarray*}
 &&Q_{1}^2=0\\
& & Q_{1}Q_{2}(v,u) + Q_{2}(Q_{1} v,u)+(-1)^{|v|}Q_{2}(v,Q_{1} u)=0\\
 & & Q_{2}(Q_{2}(v,u),w)+(-1)^{|v|}Q_{2}(v,Q_{2}(u,w))+(-1)^{(|v|+1)|u|}Q_{2}(u,Q_{2}(v,w))=0,
\end{eqnarray*}
for all $v,u,w\in s^{-1}V$.
\end{ex}

\begin{ex}
Let $(V,Q_{V})$ and $(W,Q_{W})$ be two Loday $\infty$-algebras. The direct sum $V\oplus W=\bigoplus_{i\in\Zz}(V_{i}\oplus W_{i})$ is a Loday $\infty$-algebra with $Q=Q_{V}+Q_{W}$
given by, for each $k\in\Nn$,
$$
Q(v_{1}+w_{1}, \ldots, v_{k}+w_{k})=Q_{V}(v_{1}, \ldots, v_{k}) + Q_{W}(w_{1}, \ldots, w_{k}),
$$
for all $\ds v_{1}+w_{1}, \ldots, v_{k}+w_{k}\in s^{-1}(V\oplus W)=s^{-1}V \oplus s^{-1}W.$
\end{ex}

\begin{ex}\label{ex:Lie:is:Loday}
Let $(E, M_{E}\equiv \set{l_{k}}_{k\in\Nn})$ be a Lie$[1]$ $\infty$-algebra. 
The symmetric brackets $ \set{l_{k}}_{k\in\Nn}$ define $M_{E}^{Z}$, a coderivation of $T^{Z}(E)$.
% denoted by $Q_{E}^{Z}$ such that $Q_{E}=\pi\smalcirc Q_{E}^{Z}\smalcirc N$. 
 Notice that for symmetric brackets 
Equations (\ref{eq:def:Lie:infty:algebra}) and (\ref{eq:def:Loday:infty:algebra}) are equivalent and we conclude that $M_{E}^{Z} $
 is also a codifferential and defines a Loday[1] $\infty$-algebra structure on $E$. 
Obviously this also means that any Lie $\infty$-algebra structure on a graded vector space defines a Loday $\infty$-structure on that space. \end{ex}

\begin{defn} Let $(E, Q_{E})$ and $(V,Q_{V})$ be Loday $\infty$-algebras. A \textbf{Loday $\infty$-morphism} $F: (E,Q_{E})\to (V,Q_{V})$ is a comorphism 
$F:T^{Z}(s^{-1}E)\to T^{Z}(s^{-1}V)$ such that
\begin{equation*}
F\smalcirc Q_{E}=Q_{V}\smalcirc F.
\end{equation*}
\end{defn}

\begin{rem}\label{rem:def:Loday:morphism}
A Loday $\infty$-morphism $F:(E,Q_{E}\equiv \set{l_{k}}_{k\geq 1})\to (V,Q_{V}\equiv \set{m_{k}}_{k\geq 1})$ defines and is defined by a family of degree $0$ linear maps $F_{k}: T^{k} (s^{-1}E)\to s^{-1}V$, $k\geq 1$ satisfying:
\begin{eqnarray}
&&\sum_{k=1}^{n} \sum_{\substack{i= 0\\ \sigma\in Sh(i,k-1)}}^{n-k} \!\!\!\!\! (-1)^{|v_{\sigma(1)}| +\ldots + |v_{\sigma(i)}|}\epsilon(\sigma)
F_{{n-k+1}}\left(v_{\sigma(1)} ,\ldots, v_{\sigma(i)}, \right.\nonumber\\
&&\quad\quad\quad\quad\quad\quad \left.l_{k}(v_{\sigma(i+1)},\ldots, v_{\sigma(k+i-1)}, v_{i+k} ), v_{i+k+1}, \ldots , v_{n}\right) \nonumber\\
&&=\sum_{k=1}^{n} \sum_{\substack{\sigma\in \widetilde{Sh}(i_{1},\ldots, i_{k})\\i_{1}+\ldots+i_{k}=n}} 
\epsilon(\sigma)m_{k}\left(F_{i_{1}}(v_{\sigma{(1)}}, \ldots , v_{\sigma(i_{1})}) , F_{i_{2}}(v_{\sigma{(i_{1}+1)}},\ldots , v_{\sigma(i_{1}+i_{2})}),\right. \label{eq:def:Loday:comorphism}\\
&&\quad\quad\quad\quad\quad\quad \left. \ldots , F_{i_{k}}(v_{\sigma{(i_{1}+\ldots i_{k-1}+1)}}, \ldots , v_{\sigma(n)}) \right),\nonumber
\end{eqnarray}
for all $v_1,\ldots,v_n\in s^{-1}E$.

By isomorphism (\ref{eq:isomorphism:infty:infty1}), we can equivalently define a Loday $\infty$-morphism by a family of linear maps $f_{k}: T^k(E)\to V$ (of degree $1-k$) that satisfy a collection of equations equivalent to (\ref{eq:def:Loday:comorphism}). %We will not reproduce these equations here because they will not take an important role in this work, but they can be found, for instance, in \cite{AP2010} . 
\end{rem}

\begin{rem} \label{rem:symmetric:Loday:are:Lie}
Any (Zinbiel) comorphism $F:T^{Z}(E)\to T^{Z}(V)$ is also a comorphism between coshuffle coalgebras.
When $F:T^{Z}(E)\to T^{Z}(V)$ is a comorphism defined by a family of symmetric maps $\set{F_{k}:S^{k} (E)\to V}_{k\geq 1}$ we say it is a \emph{symmetric comorphism}. 
Hence there is a natural one-to-one correspondence between symmetric (Zinbiel) comorphisms $F:T^{Z}(E)\to T^{Z}(V)$ and comorphisms between symmetric coalgebras $F^{S}:\bar S(E)\to \bar S(V)$. This correspondence is given by:
$$
F^{S}\smalcirc \pi=\pi \smalcirc F.
$$

Looking at the formula in Remark \ref{rem:def:Loday:morphism} we conclude that
$F:(E,M_{E})\to (V,M_{V})$ is a Lie $\infty$-algebra morphism if and only if it defines a symmetric Loday $\infty$-algebra morphism between 
$(E, M_{E}^{Z})$ and $(V, M_{V}^{Z})$.

Hence  the category of Lie $\infty$-algebras is a subcategory of the category of  Loday $\infty$-algebras.
\end{rem}

\begin{rem}
Let $(V, Q_V)$ be a Loday$[1]$ $\infty$-algebra. Then $Q_V$ is also a codifferential of the coshuffle coalgebra $T^c(V)$. By the coalgebra isomorphism $T^c(V) \simeq\bar S(\mathrm{Prim V})$, $Q_V$   defines a codifferential of $\bar S(\mathrm{Prim(V)})$. This means that $\mathrm{Prim}(V)$ is a Lie$[1]$ $\infty$-algebra. 
Each Loday $\infty$-morphism $F:(E,Q_E)\to (V,Q_V)$   
is also  a coshuffle comorphism that preserves codifferentials, consequently, it induces  a Lie $\infty$-morphism between $\mathrm{Prim (E)}$ and $\mathrm{Prim(V)}$.
This means there is a functor between the category of Loday$[1]$-algebras and Lie$[1]$-algebras.  \cite{STZ21}
\end{rem}

\textbf{Notation}: In what follows, we will only use the Lie$[1]$ $\infty$-algebras and the Loday$[1]$ $\infty$-algebra approach. For this reason, in the next sections, Lie/Loday $\infty$-algebras are meant to be Lie/Loday$[1]$ $\infty$-algebras, this means we will drop the shifting symbol $[1]$.

\section{Coherent actions of Lie \texorpdfstring{$\infty$}{TEXT}-algebras}\label{sec:2}
Let $(E,M_E\equiv\set{ l_k}_{k\in\Nn})$ be a Lie $\infty$-algebra and $(V, \d)$ a cochain complex. Example \ref{ex:DGLA:End:E} describes a DGLA[1] structure in $\End(V)[1]$. Recall that
 $(V,\d)$ is a \emph{(curved) representation} of $E$ if there is a Lie $\infty$-algebra morphism 
$$
\rho:(E,M_E\equiv\set{ l_k}_{k\in\Nn})\to (\End(V)[1],\partial_{\d}, \brr{\cdot , \cdot }).
$$

Now suppose 
 $(V,M_{V})$ is a Lie $\infty$-algebra. Since $(\bar S(V), M_{V})$ is a cochain complex, we have a DGLA$[1]$ structure on $s^{-1}(\Coder\bar S(V))=(\Coder\bar S(V))[1]$ induced by $\End\bar S(V)[1]$ given by:
\begin{equation*}
\left\{\begin{array}{l}
 \partial_{M_{V}}Q%=\tilde l_1\phi
=-M_{V} \smalcirc Q + (-1)^{|Q|+1} Q\smalcirc M_{V},\\
 \brr{Q,P}%=\tilde l_2(\phi,\psi)
=(-1)^{|Q|+1}\left(Q\smalcirc P - (-1)^{(|Q|+1)(|P|+1)} P\smalcirc Q \right),
\end{array}\right.
\end{equation*}
with $Q$ and $P$ homogeneous elements of $(\Coder\bar S(V))[1]$ of degrees $|Q|$ and $|P|$, respectively.

\begin{defn}
Let
$(E,M_E\equiv\set{ l_k}_{k\in\Nn})$ and $(V,M_V\equiv\set{ m_k}_{k\in\Nn})$ be Lie ${\infty}$-algebras. A {\textbf{ Lie $\infty$-action} of $E$ on $V$} is a Lie ${\infty}$-morphism 
$$
\begin{array}{lccc}
\Phi:& (E,\set{ l_k}_{k\in\Nn}) &\to& (\Coder (\bar S(V))[1], \partial_{M_V}, \brr{\cdot,\cdot}).\\
% &x&\mapsto& \Phi_x
\end{array}
$$
\end{defn}

Any Lie $\infty$-action is determined by its degree $+1$ restriction maps:
$$\begin{array}{rccl}
\Phi^k:     & S^k(E)&\to& \Coder(\bar S(V)) \\
            &  x   & \mapsto & \Phi^k(x):=p\smalcirc \Phi(x)=\Phi_x,  \quad k\in\Nn.
\end{array}
$$ 
%In what follows and whenever there is no danger of confusion,  we will write $\Phi_x$ instead of  $\Phi^k(x)$,  for each $x\in S^k(E)$.
Furthermore, any Lie $\infty$-action $\Phi$ is completely determined by the family degree $+1$ of linear maps:
$$
\Phi^{k,n}:S^{k}(E)\times S^{n}(V)\to V, 
$$
given by
$$
\Phi^{k,n}(x;v)=\Phi^{n}_{x}(v)=p\smalcirc \Phi_x(v), \quad x\in S^{k}(V), v\in S^{n}(V), \quad k,n\in \Nn.
$$
%where $p:S(V)\to V$ denotes the projection onto $V$, the cogenerator of $\bar S(V)$.

For each $x\in \bar S(E)$ and $v\in \bar S(V)$, let $\Phi_{x;v}$ be the coderivation of $\bar S(V)$ defined by the family of maps
$$
\begin{array}{lccc}
(\Phi_{x;v})_i:& S^i(V)&\to& V\\
&w& \mapsto &\Phi^{\bullet,\bullet}(x;v.w),
\end{array}, \quad i\in\Nn.
$$

\begin{rem}
A Lie $\infty$-action is a (curved) representation of $E$ in $\bar S(V)$ but the image of all the restriction maps is contained in $\Coder (\bar S(V))[1]\subset \End \bar S(V)[1]$.
\end{rem}

\begin{ex}[Adjoint representation and adjoint action]
Let 
$\left( E, M_E \equiv \set{l_k}_{k\in\Nn} \right)$ be a Lie $\infty$-algebra. 
The \emph{adjoint representation} of $E$,  $(E, l_1)$, is  defined by the collection of degree $+1$ maps
\begin{equation*} \label{eq:adjoint:representation:algebra}
\begin{array}{rrcl}
\ad_{k}:& S^{k}(E) &\to& \End(E) \\
 & \;x_1\cdot\ldots \cdot x_k & \mapsto & \ad_{x_1 \cdot \ldots \cdot x_k} := l_{k+1}\left( x_1,\ldots, x_k, \, \, \, -\,\, \right) 
\end{array}, \quad k\geq 1.
\end{equation*}

The \emph{adjoint action} of $E$ is the Lie $\infty$-action of $E$ on itself, given by the family of linear maps (see \cite{CC2022}, \cite{MZ}): 
\begin{equation*}
\begin{array}{rrcl}
{\ad^{k,i}}:& S^{k}(E)\times S^{i}(E) &\to& E \\
 & (x ; e) & \mapsto & l_{i+k}(x,e), \quad {i, k\geq 1}.
\end{array}
\end{equation*}
\end{ex}

\begin{rem}
If we define $\Phi^{0}:=M_{V}$, then
an action is equivalent to a curved Lie $\infty$-morphism between $E$ and the graded Lie algebra $\Coder (\bar S(V))$ (compatible with the Lie $\infty$-structure in $V$) \cite{MZ}. {In this case, $\Phi=\sum_{k\geq 0}\Phi^{k}$ is called a {\emph{curved Lie $\infty$-action}}}.

In this context, we can say that, for each $v\in \bar S(V)$, the coderivation 
$\Phi^0_{v}$ is the image of $v$ by the adjoint action of $V$: $$\Phi^0_v=\ad_v.$$
\end{rem}

\begin{defn} A Lie $\infty$-action of $E$ on $V$
$$
\Phi: (E,\set{ l_k}_{k\in\Nn}) \to (\Coder (\bar S(V))[1], \partial_{M_V}, \brr{\cdot,\cdot})
$$
is said to be \textbf{coherent} if
\begin{equation}\label{eq:defn:coherent:action1}
\brr{\ad_v,\Phi_x}=0 
\end{equation}
and
\begin{equation}\label{eq:defn:coherent:action2}
 \brr{\Phi_{y;v},\Phi_x}=0,
\end{equation}
for all $x,y\in \bar S(E)$ and $v\in \bar S(V)$.
\end{defn}

\begin{rem}
In the context of curved Lie $\infty$-actions, Equation (\ref{eq:defn:coherent:action1}) can be rewritten as
\begin{equation*}\label{eq:defn:coherent:action1:alternative}
 \brr{\Phi^0_v,\Phi_x}=0 , \quad x\in \bar S(E), v\in \bar S(V).
\end{equation*}
\end{rem}
\begin{rem} If we write the bracket $\brr{\cdot , \cdot }$ in terms of the usual commutator of coderivations, we have that
Equations (\ref{eq:defn:coherent:action1}) and (\ref{eq:defn:coherent:action2}) are equivalent to 
\begin{equation*}
\brr{\ad_v,\Phi_x}_{c}=0 \quad\mbox{and}\quad \brr{\Phi_{y;v},\Phi_x}_{c}=0. 
\end{equation*}

\end{rem}

\begin{ex}\label{ex:representation:is:coherent}
Let 
$(E,M_E\equiv\set{ l_k}_{k\in\Nn})$ be a Lie ${\infty}$-algebra and $(V,\d)$ a representation of $E$
given by the Lie $\infty$-morphism $\Phi:E\to \End(V)$. % a representation on the complex $(V, \d)$. 
Notice that $m_{\bullet\geq 2}=0$ so $\ad_v=0$, $v\in V$. Since $\Phi$ is a representation, we have $\Phi^{\bullet,\bullet\geq 2}=0$ which means that $\Phi_{y;v}\smalcirc\Phi_x(w)=0$, for all $v,w\in V$, $x,y\in\bar S(E)$. This way we see that
 $\Phi$ is a coherent action. 
%When $V$ is abelian, this structure coincides with the hemisemidirect %product that appears in \cite{STZ21}. 
\end{ex}

\begin{ex}\label{ex:coherent:Lie:action}
Let $(E, \brr{\cdot , \cdot }_E)$ and $(V, \brr{\cdot , \cdot }_V)$ be Lie algebras and $\rho:E\to \Der(V)$ a representation of $E$ in the classic sense.
The representation $\rho$ is equivalent to the action $\Phi:E\to \Coder (\bar S(V))$ defined by the family of maps $\Phi^{k,n}: S^{k}(E)\times S^{n}(V)\to V$, $k,n\geq 1$, given by:
$$
\Phi^{1,1}(x;v)=\rho(x)(v),\quad x\in E, v\in V
$$
$$
\Phi^{\bullet,\bullet}=0, \quad \mbox{otherwise}.
$$
In this case, the action (representation) $\Phi$ is  coherent if and only if 
$$
\brr{\ad_v,\Phi_x}_c(w)=0, \quad w\in V,
$$
which means that 
$$
\brr{v, \Phi_x w}_V - \Phi_x \brr{v,w}_V=0,\quad x\in E, \; v, w\in V.
$$
Since $\Phi_{x}=\rho({x})$ is a derivation of $\brr{\, ,\, }_{V}$ we end up with
$$
\brr{w,\Phi_x v}_V=0, \quad x\in E, \; v, w\in V.
$$
This is equivalent to $\Phi_{x}$ being a central derivation of the Lie algebra $V$, for all $x\in E$.  
In case $V$ is abelian,  the representation is obviously coeherent.
The definition of coherent representation of Lie algebras first appeared  in \cite{TS2023}.
\end{ex}

\begin{rem}
If $E$ is a Lie $\infty$-algebra, even when $V$ is abelian, that is, $m_{k}=0$, for $k>1$, the coherence of a Lie $\infty$-action $\Phi$ is not guaranteed due to the condition (\ref{eq:defn:coherent:action2}):
$$
\brr{\Phi_{x;v},\Phi_y}_c=0 \quad x,y\in \bar S(E), \;v\in \bar S(V).
$$
Clearly, if $\Phi^{\bullet, \bullet\geq 2}=0$ then coherence is guaranteed, but this means that $\Phi$ is, in fact, a Lie $\infty$-representation. 
\end{rem}

The next theorem is one of our main results. 
We show that the coherence of a Lie $\infty$-action is a necessary and sufficient condition for the construction of  a particular Loday $\infty$-algebra in the direct sum of the Lie $\infty$-algebras: the non-abelian hemisemidirect product.
This theorem generalizes two results:  Proposition 6.3 in \cite{STZ21} on  representations of Lie $\infty$-algebras and  Proposition 2.7 in \cite{TS2023} on  coherent representations of Lie algebras.

\begin{thm}\label{thm:loday:brackets:sum}
Let $\Phi: (E,\set{ l_k}_{k\in\Nn}) \to (\Coder (\bar S(V))[1], \partial_{M_V}, \brr{\cdot,\cdot})$ be a Lie $\infty$-action. Consider $E\oplus V$ equipped with brackets $\set{\mathfrak{l}_n}_{n\in \Nn}$ defined by:
\begin{eqnarray}
\lefteqn{ \mathfrak{l}_n\left(x_1+v_1,\ldots,x_{n-1}+v_{n-1}, x_n+v_n\right)=  l_n(x_1,\dots, x_n)} \nonumber\\
&&\quad +\sum_{i=1}^{n-1}\Phi^{i,n-i}(x_1,\ldots, x_{i}; v_{i+1},\dots , v_n) \label{eq:loday:brackets:sum} 
 + m_n(v_1,\dots, v_n) 
\end{eqnarray}

\begin{comment}
\begin{eqnarray*}
 \mathfrak{l}_n\left(x_1+v_1,\ldots,x_{n-1}+v_{n-1}, x_n+v_n\right)&= & l_n(x_1,\dots, x_n) \nonumber\\
 \mathfrak{l}_n\left(x_1,\ldots,x_{i},v_{i+1},\ldots, v_n\right)&=\Phi^{i,n-i}(x_1,\ldots, x_{i}; v_{i+1},\dots , v_n)  \\
 \mathfrak{l}_n\left(v_1,\ldots, v_n\right)&= m_n(v_1,\dots, v_n) \nonumber
\end{eqnarray*}

all the other kind of brackets vanish.
\end{comment}

Then $E\oplus V$ is a Loday $\infty$-algebra if and only if $\Phi$ is a coherent action.
\end{thm}

%% PROOF OF THEOREM %%%%%
\begin{proof}
Let 
$Q$ be the coderivation of the Zinbiel coalgebra $T^{Z}(E\oplus V)$ defined by the brackets $\set{\mathfrak{l}_n}_{n\in\Nn}$.
Proving that $Q^2=0$ is equivalent to showing that the Jacobiator vanishes:
$$J(x_1+v_1,\dots, x_n+v_n)=\sum_{k=1}^{n} \mathfrak l_k\smalcirc Q\left(x_1+v_1,\dots, x_n+v_n\right)=0$$
for all $x_1+v_1,\dots, x_n+v_n\in E\oplus V$.
Therefore we want to prove that coherence of the action is equivalent to the vanishing of   the Jacobiator.

We immediately see that
\begin{align*}
 J(x_1,\dots,x_n)=J(v_1,\dots,v_n)=0,\quad x_1,\dots,x_n \in E, v_1,\dots,v_n\in V,
\end{align*}
because $\set{l_k}_{k\in \Nn}$ and $\set{m_k}_{k\in \Nn}$ define Lie $\infty$-structures on $E$ and $V$, respectively.

Also, since $$\mathfrak{l}_k(x_1+v_1,\dots, x_{k-1}+v_{k-1}, x_k)=l_k(x_1,\ldots, x_k), \quad k\in \Nn,$$ we have that
$$J(x_1+v_1,\dots, x_{n-1}+v_{n-1}, x_n)=J(x_1,\ldots,x_n)=0,$$
for all $x_1,\dots,x_n \in E, v_1,\dots,v_{n-1}\in V.$

It remains to prove $J(x_1+v_1,\dots, x_{n-1}+v_{n-1}, v_n)=0$ if and only if $\Phi$ is a coherent action.

%Let $v_1,\ldots, v_n, w\in T(V)$, $y_1,\ldots, y_n\in T(E)$.

%Remember that $E\oplus V$ is a Lie $\infty$-algebra when equipped with brackets obtained by symmetrization of $\mathfrak l$. With this in mind we easily see that 

In what follows, we will use the Sweedler notation for the coshuffle coproduct: $$\Delta^{c}(v)=v_{(1)}\otimes v_{(2)},\quad v\in\bar T(V).$$
If $v\in V$, then $\Delta^c(v)=0$ and we consider $v_{(1)}=v_{(2)}=0$. 

First notice that, for $y\in \bar T(E)$ and $w\in\bar T(V)$, we have
$$
Q(y\otimes w)=M_{E}^{Z}(y)\otimes w + \Phi_{\pi(y)}^{Z}(w) + (-1)^{|y_{(1)}|}y_{(1)}\otimes \Phi_{\pi(y_{(2)})}^{Z}(w)+(-1)^{|y|}y\otimes M_{V}^{Z}(w).
$$
Therefore
\begin{eqnarray}
 J(y,w)&=&\Phi^{\bullet}_{\pi(M_E^{Z}({y}))}\pi(w)+m_\bullet\smalcirc \pi( \Phi^{Z}_{\pi(y)} (w)) + (-1)^{|y|}\Phi^{\bullet}_{\pi(y)} \pi(M^{Z}_V(w)) \nonumber\\
 &&\quad + (-1)^{|y_{{(1)}}|}\Phi^{\bullet}_{\pi(y_{{(1)}})}\pi(\Phi^{Z}_{\pi(y_{{(2)}})}(w)).
\end{eqnarray}
All coderivations in this equation are defined by symmetric linear maps. Also  $\pi:T^{c}(V)\to \bar S(V)$ is a comorphism, so we have:
\begin{eqnarray*}
J(y,w)&=&\Phi^{\bullet}_{M_E(\pi({y}))}\pi(w)+m_\bullet( \Phi_{\pi(y)} \pi(w)) + (-1)^{|y|}\Phi^{\bullet}_{\pi(y)} (M_V(\pi(w)))\\
  && + (-1)^{|y_{{(1)}}|}\Phi^{\bullet}_{\pi(y_{{(1)}})}\Phi_{\pi(y_{{(2)}})}(\pi(w)).
\end{eqnarray*}

 Taking into account that $\Phi$ is a Lie $\infty$-action we conclude that $J(y,w)=0$ (for this the coherence of the action is not needed).

%Recall that the Jacobiator would vanish if $\mathfrak l$ was symmetric. This means that the vanishing of brackets in $T(E\oplus V)\otimes E\subset T(E\oplus V)$ is responsible for the eventually non-vanishing of the Jacobiator. 
%Taking this into account we see that:

Now notice that for $v,w\in \bar T(V)$, $y\in\bar T(E)$, we have
\begin{eqnarray}\label{eq:Q:v:y:w}
    Q(v\otimes y\otimes w)&=& M_V^Z(v)\otimes y \otimes w + (-1)^{|v|} v\otimes M_E^Z(y) \otimes w \nonumber\\
    &&+ (-1)^{|v|+|y|}v\otimes y\otimes M_V^Z(w)+ (-1)^{|v|+|y_{(1)}|}v\otimes y_{(1)}\otimes \Phi^Z_{\pi(y_{(2)})}(w) \nonumber\\
    &&+ (-1)^{|v_{(1)}|+|y|(1+|v_{(2)}|)} v_{(1)}\otimes y\otimes \ad^Z_{\pi(v_{(2)})} (w)\\
    && + (-1)^{|v|}v\otimes \Phi^Z_{\pi(y)}(w) + (-1)^{|y|(|v|+1)}  y\otimes \ad^Z_{\pi(v)} (w) \nonumber
\end{eqnarray}
hence
\begin{eqnarray*}%\label{eq:Jacobiator:main1}
J(v,y, w)&=&(-1)^{|v|}m_\bullet(v,\Phi^{Z}_{\pi(y)} (w))+(-1)^{|y|(|v|+1)}\Phi^{\bullet}_{\pi(y)}(\ad_{\pi(v)}(\pi(w)))\\
&=&(-1)^{|v|}m_\bullet(\pi(v),\Phi_{\pi(y)} \pi(w)) + (-1)^{|y|(|v|+1)}\Phi^{\bullet}_{\pi(y)}(\ad_{\pi(v)}(\pi(w)))\\
&=&(-1)^{|v|}p\smalcirc\brr{\ad_{\pi(v)}, \Phi_{\pi(y)}}_c(\pi(w))\nonumber\\
&=&-p\smalcirc\brr{\ad_{\pi(v)}, \Phi_{\pi(y)}}(\pi(w)).\nonumber
\end{eqnarray*}

Similar calculations lead to the general formula, for  $y_{1},\ldots,y_{n}\in\bar T(E)$ and $v_{1},\ldots,v_{n},w\in \bar T(V)$:
 \begin{align*}%\label{eq:Jacobiator:main2}
 J(v_1,y_1,\ldots, v_n, y_n, w )=&(-1)^{\sum_{i=2}^{n}\sum_{j=1}^{i-1}|v_i||y_j|}(-1)^{|v|} p\smalcirc\brr{\ad_{\pi(v)},\Phi_{\pi(y)}}_c(\pi(w)),\\
=&- (-1)^{\sum_{i=2}^{n}\sum_{j=1}^{i-1}|v_i||y_j|} p\smalcirc\brr{\ad_{\pi(v)},\Phi_{\pi(y)}}(\pi(w)),
 \end{align*}
 where $v=v_1\otimes\ldots\otimes v_n$, $y=y_1\otimes \ldots \otimes y_n$ and $n\geq 2$.

Therefore, $J(v_1,y_1,\ldots, v_n, y_n, w )=0$, for all $v_1,\ldots,v_n,w\in\bar T(V)$ and $y_1,\ldots,y_n\in\bar T(E)$ if and only if
\begin{equation}\label{eq:Theorem:first:commutator}
 \brr{\ad_{\bar v}, \Phi_{\bar y}}(\bar w)=0,\quad \mbox{for all } \bar v, \bar w\in \bar S(V),\, \bar y\in\bar S(E).
\end{equation}

Now, for $x\in E$, $v,w\in \bar T(V)$, $y\in\bar T(E)$, we have
\begin{eqnarray*}
   \lefteqn{ Q(x\otimes v\otimes y\otimes w) = l_1(x) \otimes v\otimes y\otimes w + (-1)^{|x|} x\otimes Q(v\otimes y\otimes w)}\\
    &&  +  \Phi^Z_{x} (v)\otimes y\otimes w + (-1)^{|v|(|x|+1)} v\otimes  \ad^Z_{x} (y) \otimes w \\
    &&+ (-1)^{|y|(|x|+|v|+1)}y\otimes \Phi^Z_{x;\pi(v)}(w) \\
&&    + (-1)^{|v_{(1)}|(|x|+1)+|y|(|x|+|v_{(2)}|+1)}v_{(1)}\otimes y\otimes \Phi^Z_{x;\pi(v_{(2)})}(w) \\
    &&+ (-1)^{(|v|+|y_{(1)}|)(|x|+1)} v\otimes y_{(1)}\otimes \Phi^Z_{x\cdot \pi(y_{(2)})}(w) \\
    &&  +(-1)^{|v|(|x|+1)}v\otimes \Phi^Z_{\pi (x\otimes y)}(w) + (-1)^{(|v| + |y|)(1+|x|)}v\otimes y\otimes \Phi^Z_{x}(w).
\end{eqnarray*}
and 
\begin{eqnarray*}
\lefteqn{J(x,v,y,w)=(-1)^{|y||v|+|x|+|y|}\Phi^{\bullet}_{\pi(x\otimes y)}(\ad_{\pi(v)}(\pi(w)))}\\
&&+(-1)^{|v|(|x|+1)}m_\bullet(\pi(v),\Phi_{\pi(x\otimes y)} \pi(w))\\
&&+ (-1)^{|v|+|x|}\Phi_{x}^{\bullet}(\pi(v), \Phi_{\pi(y)} \pi(w)) + (-1)^{(|x|+|v|+1)|y|} \Phi_{\pi(y)}^{\bullet}(\Phi_{x;\pi(v)}(\pi(w)))\\
%=& (-1)^{(|x|+1)|v|}p\smalcirc\brr{\ad_v,\Phi_{xy}}_c + (-1)^{|v|+|x|}p\smalcirc\Phi_{x}(v) (\Phi_{y} w) + (-1)^{(|x|+|v|+1)|y|} p\smalcirc\Phi_{y}(\Phi_{x}(v)(w))\\
&=&(-1)^{|x||v|+|v|}p\smalcirc\brr{\ad_{\pi(v)},\Phi_{\pi(x\otimes y)}}_c (\pi(w))\\
&&+
(-1)^{|v|+|x|}p\smalcirc\brr{\Phi_{x;\pi(v)},\Phi_{\pi(y)}}_c(\pi(w)). \\
&=&-(-1)^{|x||v|}p\smalcirc\brr{\ad_{\pi(v)},\Phi_{\pi(x\otimes y)}} (\pi(w))
-p\smalcirc\brr{\Phi_{x;\pi(v)},\Phi_{\pi(y)}}(\pi(w)). 
\end{eqnarray*}

Taking into account Equation (\ref{eq:Theorem:first:commutator}) we conclude that
$J(x,v,y,w)=0$, for all $x\in E$, $y\in \bar T(E)$, $v,w\in \bar T(V)$ if and only if
\begin{equation}\label{eq:Theorem:second:commutator}
   \brr{\Phi_{x;\bar v},\Phi_{\bar y}}(\bar w), \quad \mbox{for all } x\in E, \; \bar y\in \bar S(E), \; \bar v, \bar w\in \bar S(V).
\end{equation}

Consider $x\in T^n(E)$, $n\geq 2$, and again $v,w\in \bar T(V)$, $y\in \bar T(E)$. Then,
\begin{eqnarray*}
    \lefteqn{Q(x\otimes v\otimes y\otimes w) = M_E^Z(x) \otimes v\otimes y\otimes w + (-1)^{|x|} x\otimes Q(v\otimes y\otimes w)}\\
    && + (-1)^{|x_{(1)}|} x_{(1)}\otimes \Phi^Z_{\pi(x_{(2)})} (v) \otimes y\otimes w + \Phi^Z_{\pi(x)} (v)\otimes y\otimes w\\
    && +  (-1)^{|x_{(1)}| + |v|(|x_{(2)} +1|)} x_{(1)}\otimes v\otimes \ad^Z_{\pi(x_{(2)})} (y) \otimes w \\
    &&+ (-1)^{|v|(|x|+1)} v\otimes  \ad^Z_{\pi(x)} (y) \otimes w\\
    && +  (-1)^{|x_{(1)}|+|y|(|x_{(2)}| + |v| +1)} x_{(1)}\otimes y\otimes \Phi^Z_{\pi(x_{(2)};v)} (w) \\
    &&+ (-1)^{|y|(|x|+|v|+1)}y\otimes \Phi^Z_{\pi(x;v)}(w) \\
    && + (-1)^{|x_{(1)}| + |v_{(1)}|(|x_{(2)}|+1)+|y|(|x_{(2)}|+   |v_{(2)}|+1)}x_{(1)}\otimes v_{(1)} \otimes y\otimes \Phi^Z_{\pi(x_{(2)};v_{(2)})}(w) \\
    &&+ (-1)^{|v_{(1)}|(|x|+1)+|y|(|x|+|v_{(2)}|+1)}v_{(1)}\otimes y\otimes \Phi^Z_{\pi(x;v_{(2)})}(w) \\
 && +  (-1)^{|x_{(1)}|+|v|(|x_{(2)}| +1)} x_{(1)}\otimes v\otimes \Phi^Z_{\pi(x_{(2)}\otimes y)} (w)\\
 &&+ (-1)^{|v|(|x|+1)}v\otimes \Phi^Z_{\pi(x\otimes y)}(w) \\
    && + (-1)^{|x_{(1)}| + (1+|x_{(2)}|)(|v|+|y_{(1)}|)}x_{(1)}\otimes v \otimes y_{(1)}\otimes \Phi^Z_{\pi(x_{(2)}\otimes y_{(2)})}(w) \\
    &&+ (-1)^{|v|(|x|+1)+|y_{(1)}|(|x|+1)} v\otimes y_{(1)}\otimes \Phi^Z_{\pi(x\otimes y_{(2)})}(w) \\
    && (-1)^{(|v| + |y|)(1+|x|)}v\otimes y\otimes \Phi^Z_{x}(w).
\end{eqnarray*}
Equation (\ref{eq:Q:v:y:w}) and 
 the brackets (\ref{eq:loday:brackets:sum})  yield

\begin{eqnarray*}
\lefteqn{J(x,v,y,w)=(-1)^{|v|+|x|} \Phi_{\pi(x);\pi(v)}^{\bullet}( \Phi_{\pi(y)}(\pi(w)))} \\
&&+ (-1)^{|x|+|y|(|v|+1)}\Phi^{\bullet}_{\pi(x\otimes y)}(\ad_{\pi(v)}(\pi(w)))\\
&& +  (-1)^{|x_{(1)}|+|y|(|x_{(2)}| + |v| +1)} \Phi^\bullet_{\pi(x_{(1)}\otimes y)}( \Phi_{\pi(x_{(2)});\pi(v)} (\pi(w)))\\
&&+ (-1)^{|y|(|x|+|v|+1)}\Phi^{\bullet}_{\pi(y)}( \Phi_{\pi(x);\pi(v)}(\pi(w))) \\
&&+ (-1)^{|v|(|x|+1)}m_{\bullet}(v ,  \Phi^Z_{\pi(x\otimes y)}(w)) \\
&&+  (-1)^{|x_{(1)}|+|v|(|x_{(2)}| +1)} \Phi^{\bullet}_{\pi(x_{(1)});\pi(v)} (\Phi_{\pi(x_{(2)}\otimes y)} (\pi(w) ))\\
&=& (-1)^{(|x|+1)|v|}p\smalcirc\brr{\ad_{\pi(v)},\Phi_{\pi(x\otimes y)}}_c (\pi(w))\\
&&
+(-1)^{|x|+|v|}p\smalcirc\brr{\Phi_{\pi(x);\pi(v)},\Phi_{\pi(y)} }_c(\pi(w))\\
&&+(-1)^{(|x_{(2)}|+1)|v| + |x_{(1)}|} p\smalcirc\brr{\Phi_{\pi(x_{(1)});\pi(v)},\Phi_{\pi(x_{(2)}\otimes y)}}_c(\pi(w)). \\
&=&- (-1)^{|x||v|}p\smalcirc\brr{\ad_{\pi(v)},\Phi_{\pi(x\otimes y)}} (\pi(w))\\
&&
-(-1)^{|x|}p\smalcirc\brr{\Phi_{\pi(x);\pi(v)},\Phi_{\pi(y)} }(\pi(w))\\
&&-(-1)^{|x_{(2)}||v| + |x_{(1)}|} p\smalcirc\brr{\Phi_{\pi(x_{(1)});\pi(v)},\Phi_{\pi(x_{(2)}\otimes y)}}(\pi(w)). \\
\end{eqnarray*}
%and, by induction, we can conclude that if the Jacobiators vanish then $\Phi$ must be a coherent action. 
Similar computations lead to the general formula: 
\begin{eqnarray*}
\lefteqn{ J(x,v_1,y_1,v_2,\ldots, v_n, y_n, w)=a (-1)^{|v|+ |x||v|} p\smalcirc\brr{\ad_{\pi(v)},\Phi_{\pi(x\otimes y)}}_c (\pi(w))}\\
&&+a (-1)^{|x|+|v|}p\smalcirc\brr{\Phi_{\pi(x);\pi(v)},\Phi_{\pi(y)}}_c(\pi(w))\\
&&+a (-1)^{|x_{(1)}| + |v| + |x_{(2)}||v|} p\smalcirc\brr{\Phi_{\pi(x_{(1)});\pi(v)},\Phi_{\pi(x_{(2)}\otimes y)}}_c(\pi(w)),
\end{eqnarray*}
for all
 $x,y_{1},\ldots, y_{n}\in \bar T(V)$, $v_{1}, \ldots, v_{n}\in \bar T(V)$, where $v=v_{1}\otimes\ldots \otimes v_{n}$,   $y=y_{1}\otimes\ldots\otimes y_{n}$ and $\ds a=(-1)^{\sum_{i=2}^{n}\sum_{j=1}^{i-1}|v_i||y_j|}$, $n\geq 2$.

An induction argument together with Equations (\ref{eq:Theorem:first:commutator}) and (\ref{eq:Theorem:second:commutator})  allows us to conclude that  $J(x,v,y,w)=0$, for all $x,y\in\bar T(E)$ and $v,w\in\bar T(V)$, if and only if
\begin{equation}\label{eq:Theorem:third:commutator}
    \brr{\Phi_{\bar x;\bar v},\Phi_{\bar y} }(\bar w),\quad \mbox{ for all } \bar x, \bar y\in\bar S(E), \bar v,\bar w\in \bar S(V).
\end{equation}

Equations (\ref{eq:Theorem:first:commutator}), (\ref{eq:Theorem:second:commutator}) and (\ref{eq:Theorem:third:commutator}) express the coherence of the action $\Phi$ and we conclude that it is equivalent to the  Jacobiator being zero.
\end{proof}

%%%%%%%%%%%%%%%%%%%%%%%%%%%%%%%%%%%%%

Following \cite{TS2023}, we define
\begin{defn} Let $(E,\set{ l_k}_{k\in\Nn})$ and $(V, M_V\equiv\set{m_k}_{k\in\Nn})$ be Lie $\infty$-algebras and
 $\Phi: (E,\set{ l_k}_{k\in\Nn}) \to (\Coder (\bar S(V))[1], \partial_{M_V}, \brr{\cdot,\cdot})$  a coherent Lie $\infty$-action. The Loday $\infty$-algebra $(E\oplus V, \set{\mathfrak l_k}_{k\in \Nn})$ given by 
\begin{eqnarray*}
\lefteqn{ \mathfrak{l}_n\left(x_1+v_1,\ldots,x_{n-1}+v_{n-1}, x_n+v_n\right)=  l_n(x_1,\dots, x_n)} \nonumber\\
&&\quad +\sum_{i=1}^{n-1}\Phi^{i,n-i}(x_1,\ldots, x_{i}; v_{i+1},\dots , v_n) 
 + m_n(v_1,\dots, v_n), %\nonumber
\end{eqnarray*}
for all $x_1+v_1, \ldots, x_n+v_n\in E\oplus V$, is called the {\textbf{non-abelian hemisemidirect product}} of $E$ and $V$ and is denoted by $E\ltimes^{\Phi} V$.
\end{defn}

\begin{ex}
  Let $(E,\brr{\cdot , \cdot }_E)$ and $(V, \brr{\cdot ,\cdot}_V)$ be Lie algebras and $\rho:E\to \Der(V)$ a coherent Lie representation (see Example \ref{ex:coherent:Lie:action}):
  $$\brr{w,\rho_x(v)}_V=0, \quad x\in E, \, v,w\in V.$$
This particular case was considered 
 in \cite{TS2023} and the non-abelian hemisemidirect product $E\ltimes^{\Phi} V $ is the Loday algebra (or the left Leibniz algebra) with  product given by
$$
(x+v)\smalcirc( y+w)= \brr{x,y}_E + \Phi_x w + \brr{v,w}_V.
$$

Notice that when the representation $V$ is an abelian Lie algebra, then it is always coherent and we obtain the usual hemisemidirect product introduced in \cite{KW2001}. 
\end{ex}

\begin{ex}
Let 
$(E,M_E\equiv\set{ l_k}_{k\in\Nn})$ be a Lie ${\infty}$-algebra 
and $\Phi:\bar S(E)\to \End(V)$ a representation on the complex $(V, \d)$. The  action $\Phi$ is coherent (see Example \ref{ex:representation:is:coherent}).

The non-abelian hemisemidirect product $E\ltimes^{\Phi} V$ is  defined by the brackets:
\begin{eqnarray*}
    \mathfrak l_1(x_1+v_1)&=&l_1 x_1 + \d\,v_1,\\
    \mathfrak{l}_n(x_1+v_1,\ldots , x_n+v_n)&=&l_n(x_1,\ldots, x_n) + \Phi^{n-1,1}(x_1,\ldots,x_{n-1}; v_n),
\end{eqnarray*}
for all $x_1+v_1,\ldots, x_n+v_n \in E\oplus V$, $n\geq 2$ .
%When $V$ is abelian, this structure coincides with the hemisemidirect %product that appears in \cite{STZ21}. 

\end{ex}

%\begin{rem} 
%Any Lie $\infty$-action $\Phi: (E,\set{ l_k}_{k\in\Zz}) \to (\Coder (\bar S(V))[1], \partial_{M_V}, \brr{\cdot,\cdot})$  gives to $E\oplus V$ a Lie $\infty$-%algebra structure $\set{\mathfrak{l}^{S}_{k}:\bar S(E\oplus V)\to E\oplus V}_{k\geq 1}$ (see \cite{CC2022}). If $\Phi$ is coherent then the brackets 
%$\set{\mathfrak{l}^{S}_{k}}$ are the symmetrization of the brackets that define the Loday $\infty$-algebra given by (\ref{eq:loday:brackets:sum}).
%\end{rem}

\begin{rem}
A Lie $\infty$-action $\Phi: (E,M_{E}\equiv\set{ l_k}_{k\in\Nn}) \to (\Coder (\bar S(V))[1], \partial_{M_V}, \brr{\cdot,\cdot})$, is a special representation of $E$ on the complex $(\bar S(V),M_V)$. Since any representation is a coherent action, we have the non-abelian hemisemidirect product $E\ltimes^{\Phi} \bar S(V)$ with brackets:
\begin{eqnarray*}
\bar{\mathfrak{l}}_1(x_{1}+v_{1})&=&l_{1}(x_{1})+ M_{V}(v_{1}),\\
\bar{\mathfrak{l}}_n(x_{1}+v_{1},\ldots, x_{n}+v_{n})&=&l_{n}(x_{1},\ldots,x_{n}) + \Phi_{x_{1}\cdot\ldots\cdot {x_{n-1}}} v_{n},
\end{eqnarray*}
for all $x_{1},\ldots,x_{n}\in E$, $v_{1},\ldots, v_{n}\in\bar S(V)$, $n\geq 2$.
%Loday $\infty$-algebra
%and $E\oplus \bar S(V)$ has a Loday $\infty$-algebra structure: $E\ltimes^{\Phi} S(V)$.

If $\Phi$ is a coherent Lie $\infty$-action, then $E\ltimes^{\Phi} V$ is also a Loday $\infty$-algebra. The brackets $\set{\mathfrak{l}_k}_{k\in\Nn}$ in $E\ltimes^{\Phi} V$ and the brackets $\set{\bar{\mathfrak{l}}_k}_{k\in\Nn}$ in $E\ltimes^{\Phi} \bar S(V)$ are related by
\begin{eqnarray*}
\mathfrak{l}_{1}(x_{1}+v_{1})&=&\bar{\mathfrak{l}}_{1} (x_{1}+v_{1})\\
 \mathfrak{l}_n(x_1+v_1,\ldots, x_n+v_n)&=& \bar{\mathfrak{l}}_n(x_1,\ldots, x_n) \\ 
 &&  \!\!\!\!\!\!\!\!\!\!\!\!\!\!\!\!\!\!\!\!\!\!\!\!\!\!\!\!\!\!+p_{V}\smalcirc\left(\sum_{i=1}^{n-1}\bar{\mathfrak{l}}_{i+1}(x_1,\ldots, x_i, v_{i+1}\cdot\ldots\cdot v_n) + \bar{\mathfrak{l}}_1(v_1\cdot\ldots\cdot  v_n)\right),
\end{eqnarray*}
for all $x_{1},\ldots,x_{n}\in E$, $v_{1},\ldots, v_{n}\in V$, $n\geq 2$, where $p_V:\bar S(V)\to V$ is the projection map. 
\end{rem}

%%%%%%%%%%%%%%%%%%%%%%%%%%%%%%%%%%%%%%%%%%%%%%%%%%%%%

%%%%%%%%%%%%%%%%%%%%%%%%%%%%%%%%%%%%%%%%%%%%%%%%%%%%%

\section{Non-abelian embedding tensors}\label{sec:3}

Let 
$(E,M_E\equiv\set{ l_k}_{k\in\Nn})$ and $(V,M_V\equiv\set{ m_k}_{k\in\Nn})$ be Lie ${\infty}$-algebras, 
$$
\Phi: (E,\set{ l_k}_{k\in\Nn}) \to (\Coder (\bar S(V))[1], \partial_{M_V}, \brr{\cdot,\cdot})
$$
a coherent Lie $\infty$-action and $E\ltimes^{\Phi} V$ the non-abelian hemisemidirect product defined by $\Phi$. 

The graded vector space of linear maps between $\bar T(V)$ and $E$ will be denoted by $\mathfrak{h}:=\Hom(\bar T(V),E)$. It can be identified with the space of coalgebra morphisms between $T^{Z}(V)$ and $T^{Z}(E)$ and also with a subspace of $\Coder ({T}^{Z}(E\oplus V))$.

%Let $t:\bar T(V)\to E$ be an element of $\mathfrak{h}$ defined by the collection of maps
%$t_k: T^{k}(V)\to E$, $k\geq 1$. Let us denote by $T: T^{Z}(V)\to T^{Z}(E)$ the coalgebra morphism and by $\mathfrak{t}$ the coderivation of $T^{Z}(E\oplus V)$ defined by $t$. 
% Remember that, in particular, $$\mathfrak{t}(v)=t_1(v), \quad v\in V$$ 
%and
%$$ \mathfrak{t}(v_1\otimes\ldots \otimes v_n)=\sum_{k=1}^{n}\!\!\!\!\!\!\!\!\sum_{\begin{array}{c}\scriptstyle{i=1} \\ \scriptstyle{\sigma\in Sh(i,k-i-1)}\end{array}}^{k-1} \!\!\!\!\!\!\!\!\!\!\!v_{\sigma{(1)}}\otimes\ldots \otimes v_{\sigma{(i)}}\otimes t_{k-i}(v_{\sigma{(i+1)}},\ldots, v_{\sigma{(k-1)}}, v_k)\otimes v_{k+1}\otimes \ldots \otimes v_n,$$
%for $v_1,\ldots, v_n\in V$.

The space $\mathfrak{h}$ together with the coherent action defines a $V$-data \cite{V05} and, consequently, $\mathfrak{h}$ acquires a  Lie $\infty$-algebra structure:

\begin{prop}
 Let us consider:
\begin{enumerate}
 \item the graded Lie algebra $\mathcal{L}:=(\Coder(T^{Z}(E\oplus V)), \brr{\cdot , \cdot }_{c})$;
 \item the abelian graded Lie subalgebra $\mathfrak{h}$;
 \item the projection $\mathcal{P}:\mathcal{L}\to \mathfrak{h}$ onto $\mathfrak{h}$;
 \item  the coderivation $Q$ associated with the Loday ${\infty}$-algebra $E\ltimes^{\Phi} V$.
\end{enumerate}
Then, $(\mathcal{L}, \mathfrak{h}, \mathcal{P},Q)$ is a $V$-data and
 $\mathfrak{h}$ is equipped with a Lie $\infty$-algebra structure given by:
\begin{align*}
\partial_{k}\big(\mathfrak t_{1},\ldots, \mathfrak t_{k})=\mathcal{P}([[\ldots\brr{Q, \mathfrak{t}_{1}}_{c}\ldots]_{c}, \mathfrak{t}_{k}]_{c}\big), \quad \mathfrak t_{1},\ldots, \mathfrak t_{k}\in\mathfrak{h}, \, k\geq 1.
\end{align*}
\end{prop}

\begin{proof}
Although we are dealing with Lie $\infty$-actions rather than representations, the argument used in the proof of  Proposition 6.4 in \cite{STZ21}, to show that we have a $V$-data, can be used in an identical way in this case.
\end{proof}

Now we are in position to define a non-abelian embedding tensor with respect to a coherent Lie $\infty$-algebra action $\Phi:E\to \Coder(\bar S(V))$.

\begin{defn}
A \textbf{non-abelian embedding tensor} on $E$ with respect to the coherent action $\Phi$ 
is a (Zinbiel) coalgebra morphism $T\equiv\set{T_{k}:T^{k}(V)\to E}_{k\in \Nn}$, such that the induced coderivation $\mathfrak{t}\in\mathfrak{h}$ is a Maurer-Cartan element of $\mathfrak{h}$:
\begin{equation}\label{def:embedding:Maurer:Cartan:element}
\mathcal{P}\left(\brr{Q,\mathfrak{t}}_{c} + \frac{1}{2}\brr{\brr{Q,\mathfrak t}_{c},\mathfrak{t}}_{c} + \dots \right)=0.
\end{equation}
\end{defn}

\begin{rem} By setting  $$e^{\brr{\, , \mathfrak t}}(Q)=Q + \brr{Q, \mathfrak{t}}_c  + \frac{1}{2}\brr{\brr{Q, \mathfrak{t}}_c, \mathfrak{t}}_c + \ldots + \frac{1}{n!} \brr{\ldots \brr{\brr{Q, \mathfrak{t}}_c,\mathfrak{t}}_c \ldots  \mathfrak{t}}_c  +\ldots$$
 we can rewrite the Equation (\ref{def:embedding:Maurer:Cartan:element})  as
\begin{equation*}
 \mathcal{P}\left(e^{\brr{\, , \mathfrak t}}(Q) \right)=0.
\end{equation*}
\end{rem}

Let $T: T^{Z}(V)\to T^{Z}(E)$ be
a non-abelian embedding tensor with respect to the coherent action $\Phi$. Suppose $T$ is defined by the family of maps $\set{T_{k}:T^{k} (V)\to E}_{k\geq 1}$ and
consider $\mathfrak t\in\Coder( T^{Z}(E\oplus V))$ the associated coderivation. 

The projection $p_V:E\oplus V \to V$ defines the strict comorphism 
 $p_{T(V)}:T^Z(E\oplus V)\to T^Z(V)$. The composition $T\smalcirc p_{T(V)}$ is the comorphism defined by the
   family of linear maps $\set{T_k\smalcirc p_{T(V)}: T^k(E\oplus V)\to E}_{k\geq 1}$. This composition  can be see  as an extension of $T$ to $\bar T(E\oplus V)$ so  we also denote it by $T$.

\begin{lem}\label{lem:exponential:coderivation}
The exponential of the coderivation $\mathfrak{t}$ is the comorphism $\mathrm{Id}+T:T^{Z}(E\oplus V)\to T^{Z}(E\oplus V):$ 
$$
e^{\mathfrak{t}}=\sum_{k=0}^{\infty}\frac{\mathfrak{t}^{k}}{k!}=\mathrm{Id} + T.
$$
\end{lem}

\begin{proof}
The exponential of a (conilpotent) coderivation of degree $0$ is a comorphism. The restriction maps of $e^{\mathfrak{t}}$ are given by
\begin{eqnarray*}
(e^{\mathfrak t})_{1}(x+v)&=&x+v+T_{1}(v),\quad x\in E, v\in V \\
(e^{\mathfrak t})_{n}(x_{1}+v_{1},\ldots, x_{n}+v_{v})&=&T_{n}(v_{1},\ldots, v_{n}), \quad n\geq 2,
\end{eqnarray*}
for all $x_{1}+v_{1},\ldots, x_{n}+v_{n}\in E\oplus V$, and this yields the result.
\end{proof}

\begin{lem}\label{lem:Loday:bracket:formula} %For each $n\in\Nn_{0}$,
%The map $\ds p_{T(V)}Q\frac{\mathfrak{t}^{n}}{n!}$ is a coderivation of $T^{Z}(E\oplus V)$ with respect to the comorphism $p_{T(V)}$.
The map $p_{T(V)}Qe^{\mathfrak{t}}$ is a coderivation with respect to the comorphism $p_{T(V)}$. Its restriction to $T(V)$, $p_{T(V)}Qe^{\mathfrak{t}}_{|_{TV}}$, is defined by the restriction maps:
\begin{equation*}
p_{V}Qe^{\mathfrak{t}}(v_{1})=m_{1}(v_{1}),
\end{equation*}
\begin{equation}
    \label{eq:loday:brackets:in:V}
p_{V}Qe^{\mathfrak{t}}(v_{1}\otimes\ldots\otimes  v_{n})=m_{n}(v_{1},\ldots, v_{n})
 + \sum_{k=1}^{n-1}\Phi^{\bullet}_{\pi T(v_{1}\otimes\ldots\otimes v_{k})}( v_{k+1},\ldots,v_{n} )
\end{equation}
for all $v_{1},\ldots,v_{n}\in  V$, $n\geq 2$.
\end{lem}

\begin{proof} The projection map $p_{T(V)}:T^{Z}(E\oplus V)\to T^{Z}(V)$ is a comorphism. Since $e^{\mathfrak{t}}$ is also a comorphism, the composition 
$p_{T(V)}Qe^{\mathfrak{t}}$ is a coderivation with respect to the comorphism 
 $p_{T(V)}\smalcirc e^{\mathfrak{t}}=p_{T(V)}\smalcirc(\mathrm{Id}+T)=p_{T(V)}$. In fact,
 \begin{eqnarray*}
\Delta^{Z}p_{T(V)}Qe^{\mathfrak{t}}&=&p_{T(V)}\otimes p_{T(V)}\left(\mathrm{Id} \otimes Q+ Q\otimes \mathrm{Id}\right)(e^{\mathfrak{t}}\otimes e^{\mathfrak{t}})\Delta^{Z}\\
&=&\left(p_{T(V)}e^{\mathfrak{t}} \otimes (p_{T(V)}Qe^{\mathfrak{t}})+ (p_{T(V)}Qe^{\mathfrak{t}})\otimes p_{T(V)}e^{\mathfrak{t}}\right)\Delta^{Z}\\
&=&\left(p_{T(V)} \otimes (p_{T(V)}Qe^{\mathfrak{t}})+ (p_{T(V)}Qe^{\mathfrak{t}})\otimes p_{T(V)}\right)\Delta^{Z}, 
\end{eqnarray*}
which means that $p_{T(V)}Qe^{\mathfrak{t}}$ is a $p_{T(V)}$-coderivation.

Let $v_{1},\ldots, v_{n}\in V$. For $n=1$, the restriction map of $p_{T(V)}Qe^{\mathfrak{t}}_{|_{TV}}$ is given by:
$$
p_{V}Qe^{\mathfrak{t}}(v_{1})=p_{V}Q(\mathrm{Id}+T)(v_{1})=m_{1}(v_{1}).
$$
For $n\geq 2$, the restriction map of $p_{T(V)}Qe^{\mathfrak{t}}_{|_{TV}}$ is given by:
\begin{eqnarray*}
p_{V}Qe^{\mathfrak{t}}(v_{1}\otimes\ldots\otimes  v_{n})&=& p_{V}Q(\mathrm{Id}+T)(v_{1}\otimes\ldots\otimes  v_{n}) \\
 &=&p_{V}Q\!\!\!\!\!\!\!\!\!\!\sum_{\substack{i_{1}+\ldots+i_{k}=n\\ \sigma\in\widetilde{Sh}(i_{1},\ldots,i_{k})}} \!\!\!\!\!\!\!\!\!\!\epsilon(\sigma)(\mathrm{Id} + T)_{i_{1}}( v_{\sigma(1)}, \ldots , v_{\sigma(i_{1})} )\otimes \ldots \\
&&\quad\quad \ldots\otimes (\mathrm{Id} + T)_{i_{k}}(v_{\sigma(i_{1}+\ldots +i_{k-1}+1)}, \ldots ,v_{\sigma(n)} ).
\end{eqnarray*}
%for all $v_{1},\ldots, v_{n}\in V$.

Now Equation (\ref{eq:loday:brackets:in:V}) follows if we keep in mind three things: 

(i) The restriction maps of the identity are given by: $\mathrm{Id}_{1}=\mathrm{Id}_{E\oplus V}$ and $\mathrm{Id}_{k}=0$, for $k>1$. 

(ii)
$\sigma\in\widetilde{Sh}(i_{1},\ldots,i_{k})$ means that
$\sigma(i_{1})<\sigma(i_{1}+i_{2})<\ldots <\sigma(i_{1}+\ldots+i_{n-1})<\sigma(i_{1}+\ldots+i_{n})=n $. 

(iii) By (\ref{eq:loday:brackets:sum}) we have
\begin{eqnarray*}
&&p_{V}Q(v_{1}\otimes \ldots\otimes v_{n})=m_{n}(v_{1},\ldots, v_{n}),\\
&& p_{V}Q(x_{1}\otimes\ldots\otimes x_{k}\otimes v_{1}\otimes\ldots \otimes v_{p})=\Phi^{\bullet,\bullet}(x_{1},\ldots, x_{k}; v_{1},\ldots, v_{p}), \quad x_{1},\ldots,x_{k}\in E,\\
&& p_{V}Q=0 \mbox{ in all the other cases.} 
\end{eqnarray*}

Therefore,
\begin{eqnarray*}
p_{V}Qe^{\mathfrak{t}}(v_{1}\otimes\ldots\otimes v_{n})= m_{n}(v_{1},\ldots, v_{n})
  + \sum_{k=1}^{n-1}\Phi^{\bullet}_{\pi T(v_{1}\otimes\ldots\otimes v_{k})}( v_{k+1},\ldots,v_{n} ).
\end{eqnarray*}
\end{proof}

Next proposition gives an explicit description of non-abelian embedding tensors.

\begin{prop} \label{prop:embedding:tensor:explicitly}
Equation (\ref{def:embedding:Maurer:Cartan:element}) is equivalent to
\begin{equation*}
 l_1\smalcirc T_1= T_1\smalcirc m_1
\end{equation*}
and \begin{eqnarray*}
\lefteqn{l_{\bullet}\left(T(v)\right)=T_{\bullet}\left(M^{Z}_{V}(v)\right)
 + \sum_{k=2}^{n}\sum_{i=0}^{k-2}
\sum_{\substack{j=i+1\\ \sigma\in Sh(i,k-i-1)}}^{k-1} \epsilon(\sigma)(-1)^{|v_{{\sigma(1)}}|+\ldots+ |v_{{\sigma(i)}}|}}\\
&& T_{\bullet}\left(v_{{\sigma(1)}}, \ldots, v_{{\sigma(i)}},
\Phi^{\bullet}_{\pi T(v_{{\sigma(i+1)}}\otimes \ldots\otimes v_{{\sigma(j)}} )} \left(v_{{\sigma(j+1)}}, \ldots,v_{{\sigma(k-1)}},v_{{k}} \right), v_{{k+1}}, \ldots, v_{{n}}\right),
\end{eqnarray*}
for all $v=v_{1}\otimes\ldots\otimes v_{n}\in T^{n}(V)$, $n\geq 2$.
%\\ \sigma(i)<\sigma(k-1)
\end{prop}

\begin{proof} First notice that 
$$p_{E}\mathfrak{t}^{k}=0, \quad k\geq 2.$$
Thus
$$
p_{E}(e^{\brr{\, , \mathfrak{t}}}(Q))=p_{E}(Qe^{\mathfrak{t}} - \mathfrak{t}Qe^{\mathfrak{t}})
$$
and Equation (\ref{def:embedding:Maurer:Cartan:element}) is equivalent to
$$
p_{E}Qe^{\mathfrak{t}}_{|_{T(V)}} =p_{E} \mathfrak{t}Qe^{\mathfrak{t}}_{|_{T(V)}}
$$
or to
$$
l_{\bullet}p_{T(E)}e^{\mathfrak{t}}_{|_{T(V)}}=T_{\bullet}p_{T(V)}Qe^{\mathfrak{t}}_{|_{T(V)}}.
$$

\

Let $v_{1},\ldots, v_{n}\in V$. 
By Lemma \ref{lem:exponential:coderivation}, we have
$$p_{E}Qe^{\mathfrak{t}}(v_{1}\otimes\ldots\otimes v_{n})=l_{\bullet}p_{T(E)}e^{\mathfrak{t}}(v_{1}\otimes\ldots\otimes v_{n})=l_{\bullet}T(v_{1}\otimes\ldots\otimes v_{n}).$$

On the other hand, by Lemma \ref{lem:Loday:bracket:formula} we know that 
$p_{T(V)}Qe^{\mathfrak{t}}_{|_{T(V)}}$ is the $p_{T(V)}$-coderivation defined by the linear maps (\ref{eq:loday:brackets:in:V}). This yields
\begin{eqnarray*}
\lefteqn{p_{T(V)}Qe^{\mathfrak{t}}(v_{1}\otimes\ldots\otimes v_{n})=M_{V}^{Z}(v_{1}\otimes\ldots\otimes v_{n})} \\
&& + \sum_{k=2}^{n}\sum_{i=0}^{k-2}\!\!\!\!\!\!\!\!\!\!
\sum_{\substack{j=i+1\\ \sigma\in Sh(i,k-i-1)}}^{k-1}\!\!\!\!\!\!\!\!\!\! \epsilon(\sigma)(-1)^{(|v_{{\sigma(1)}}|+\ldots+ |v_{{\sigma(i)}}|)}
v_{{\sigma(1)}}\otimes \ldots\otimes v_{{\sigma(i)}}\otimes \\
&&\qquad\otimes \Phi^{\bullet}_{\pi T(v_{{\sigma(i+1)}}\otimes \ldots\otimes v_{{\sigma(j)}} )} \left(v_{{\sigma(j+1)}}, \ldots,v_{{\sigma(k-1)}},v_{{k}} \right)
\otimes v_{{k+1}}\otimes \ldots\otimes v_{{n}}.
\end{eqnarray*}
\end{proof}

Non-abelian embedding tensors define a Loday $\infty$-structure on $V$. 

\begin{prop}\label{prop:embedding:tensor:morphism}
A non-abelian embedding tensor $T$ with respect to a coherent Lie $\infty$-algebra action $\Phi:(E,M_E) \to (\Coder (\bar S(V))[1], \partial_{M_V}, \brr{\cdot, \cdot })$ induces on $V$ a Loday $\infty$-algebra structure $Q_{V}^{T}$, 
given by the brackets 
\begin{align*}
q_1(v_{1})&=m_1(v_{1})\\
 q_{n}(v_{1},\ldots,v_{n}) &= \sum_{k=1}^{n-1}\Phi^{\bullet}_{\pi T(v_{1}\otimes\ldots\otimes v_{k})}\left(v_{k+1},\ldots, v_{n}\right) + m_n(v_{1},\ldots,v_{n}), 
\end{align*}
for all $v_{1},\ldots,v_{n} \in V$, $n\geq 2$.

Moreover, $T:(V,Q^{T}_{V})\to (E, M_{E}^{Z})$ is a Loday $\infty$-morphism. 
\end{prop}

\begin{proof} Let $Q$ be the codifferential (of degree $+1$) of the non-abelian hemisemidirect direct product $E\ltimes^{\Phi} V$. Since
$$
e^{\brr{\, , \mathfrak t}}(Q)\smalcirc e^{\brr{\, , \mathfrak t}}(Q)=e^{\brr{\, , \mathfrak t}}(Q^2)=0,
$$
also $e^{\brr{\, , \mathfrak t}}(Q)$ %=e^{-\mathfrak{t}}Qe^{\mathfrak{t}}$ 
is a coderivation of degree $+1$ of $T^{Z}(E\oplus V)$ and, consequently, 
 it defines Loday $\infty$-structure on $E\oplus V$.
 %:
%$$
%Q_{E\oplus V}^{T}=e^{\brr{\, , \mathfrak t}}(Q).
%$$

By definition $T$ is a non-abelian embedding tensor if 
$\ds 
\mathcal{P}\left(e^{\brr{\, , \mathfrak t}}(Q) \right)=0,
$
therefore
\begin{equation*}%\label{eq:embedding:Leibniz:infty:V}
p_{V}\left(e^{\brr{\, , \mathfrak t}}(Q) \right)_{|_{{T(V)}}}
\end{equation*}
 defines a Loday $\infty$-structure on $V$. 
Noting that $\ds p_{V}\smalcirc \mathfrak{t}=0$, 
we have $p_{V}\left(e^{\brr{\, , \mathfrak t}}(Q) \right)=p_{V}\smalcirc Q\smalcirc e^{\mathfrak{t}}$ and the Lemma \ref{lem:Loday:bracket:formula} gives the Loday $\infty$-brackets
\begin{eqnarray*}
q_{1}(v_{1})&=&m_{1}(v_{1})\\
q_{n}({v_{1},\ldots, v_{n}})&=& m_{n}(v_{1},\ldots, v_{n}) + \sum_{k=1}^{n-1}\Phi^{\bullet}_{\pi T(v_{1}\otimes\ldots\otimes v_{k})} (v_{k+1},\ldots, v_{n}),
\end{eqnarray*}
for all $v_{1},\ldots,v_{n}\in V$, $n\geq 2$.

Finally, let $Q^{T}_{V}=p_{T(V)}Qe^{\mathfrak{t}}_{|_{T(V)}}$ be the codifferential of $T^{Z}(V)$ defined by these brackets.  Proposition \ref{prop:embedding:tensor:explicitly} guarantees that
 $T$ is a non-abelian embedding tensor if and only if
$M_{E}^{Z}\smalcirc T=T\smalcirc Q^{T}_{V}$. 
This means that $T$ is a Loday $\infty$-morphism.
\end{proof}

Following the nomenclature used in \cite{TS2023}:
\begin{defn} The Loday $\infty$-structure on $V$ given by a non-abelian embedding tensor with respect to a coherent Lie $\infty$-action is called the {\textbf{descendent Loday $\infty$-algebra}}.
\end{defn}

\begin{cor} Let $(E,M_{E})$ and $(V,M_{V})$ be Lie $\infty$-algebras and
 $T$ a symmetric non-abelian embedding tensor with respect to a coherent action $\Phi$ such that $\pi(\mathrm{Im} \;T)\subset \ker \Phi$, then $ T$ defines a Lie $\infty$-morphism. 
\end{cor}
\begin{proof}
Let $T$ be a non-abelian embedding tensor with respect to an coherent action $\Phi$ such that $\pi(\mathrm{Im} \;T)\subset \ker \Phi$. Then, by Proposition \ref{prop:embedding:tensor:morphism},
$$
M_{E}^{Z}\smalcirc T=T\smalcirc M_{V}^{Z}.
$$

The (Zinbiel) comorphism $T\equiv\set{T_k: T^k(V)\to E}_{k\in\Nn}$ is symmetric if $T_{\bullet}=T_{\bullet}\smalcirc \pi$ (see Remark \ref{rem:symmetric:Loday:are:Lie}). 
In this case, the same family of linear maps defines a comorphism $T^S :\bar S(V)\to \bar S(E)$ that is a Lie $\infty$-morphism because $T$ is a Loday $\infty$-morphism.
\end{proof}

\begin{cor}
    Let $(E,M_{E})$ and $(V,M_{V})$ be Lie $\infty$-algebras and
 $T$ a  non-abelian embedding tensor with respect to a coherent action $\Phi$.
 Then $T$ induces a Lie $\infty$-morphism 
$ \bar T:\mathrm{Prim}(V)\to E$. 
\end{cor}
\begin{proof}
    The non-abelian embedding tensor  is a Loday $\infty$-morphism 
    $$T:(V,Q_V^T)\to (E,M_E^Z),$$ 
    therefore it is also  a comorphism between coshuffle coalgebras and preserves the codifferentials.

    Recall that, for any graded vector space $V$, there is a (coalgebra) isomorphism $\Psi: \bar T^c(V)\to \bar S(\mathrm{Prim}(V))$ (see \cite{STZ21}). 
    
    In our case, $(V,Q^T_V)$  and $(E, M_E^Z)$ are  Loday $\infty$-algebras so  we can use the isomorphism and the codifferentials $Q_V^T$ and $M_E^Z$ to turn $\mathrm{Prim}(V)$ and $\mathrm{Prim}(E)$ into Lie $\infty$-algebras: $$ Q_V^{T,\, Lie}:=\Psi\smalcirc Q_V^T\smalcirc\Psi^{-1},$$
    $$ M_E^{Lie}:=\Psi\smalcirc M_E^Z\smalcirc\Psi^{-1}.$$
   
We have the following diagram:

\begin{center}
    \xymatrix{
(V, Q_V^T)\subset T^c(V) \ar[d]^{\Psi} \ar[r]^T & (E,M_E^Z)\subset T^c(E)\ar[r]^{\pi}\ar[d]^{\Psi} & (E,M_E)\subset \bar S(E)\\
(\mathrm{Prim}(V),Q_V^{T,\, Lie}) \ar@/^/[u]^{\Psi^{-1}} \ar[r]^{T^{Lie}}  & (\mathrm{Prim}(E),M_E^{Lie})\ar@/^/[u]^{\Psi^{-1}}
}    
\end{center}
      
The symmetrization map $\pi:\bar T^c(E)\to \bar S(E)$ is a comorphism such that
$$\pi\smalcirc M_E^Z=M_E\smalcirc \pi.$$
Therefore, the map $\bar T=\pi\smalcirc T\smalcirc\Psi^{-1}: \bar S(\mathrm{Prim}(V))\to \bar S(E)$ is a comorphism that preserves the codifferentials, which means it is a Lie $\infty$-algebra morphism.  
\end{proof}

\begin{ex}
In the case where $V$ is simply a vector space (i.e. $M_{V}=0$) and $\rho:(E,M_{E})\to \mathfrak{gl}(V)$ is a (non-curved) representation of the Lie $\infty$-algebra $(E,M_{E})$,  a non-abelian embedding tensor of $\rho$ is a homotopy embedding tensor of the representation $\rho$, that was defined in \cite{STZ21}.
\end{ex}

\begin{ex}\label{ex:embedding:tensor:representation}
Let $\Phi:(E,M_E=\set{l_{k}}_{k\in \Nn}) \to (\End (V), \mathrm{d})$ be a Lie $\infty$-representation. A (Zinbiel) comorphism
 $T\equiv\set{T_{k}:T^{k}(V)\to E}_{k\in\Nn}$ is a non-abelian embedding tensor if and only if
 \begin{eqnarray*}
 l_{1}\smalcirc T_{1}&=&T_{1}\smalcirc \mathrm{d}\\
 l_{\bullet}T(v_{1}\otimes\ldots\otimes v_{n})
 &=& \sum_{i=1}^n  (-1)^{|v_1|+\ldots+|v_{i-1}|}T_n(v_1\ldots, \mathrm{d}v_i, \ldots,v_n)\\
 &&+
 \sum_{k=2}^{n}\!\!\!\!\!\sum_{\substack{i=0\\ \sigma\in Sh(i,k-i-1)}}^{k-2}\!\!\!\!\!\!\!\!\!\!\epsilon(\sigma)(-1)^{(|v_{{\sigma(1)}}|+\ldots+ |v_{{\sigma(i)}}|)} \\
 &&\quad T_{\bullet}\left(v_{{\sigma(1)}}, \ldots, v_{{\sigma(i)}},
 \Phi_{\pi T(v_{{\sigma(i+1)}}\otimes \ldots\otimes v_{{\sigma(k-1)}} )} v_{{k}}, 
 v_{{k+1}}, \ldots, v_{{n}}\right),
 \end{eqnarray*} 
 for all $v_{1},\ldots, v_{n}\in V$, $n\geq 2$.

The Loday $\infty$-structure on $V$ defined by the non-abelian embedding tensor $T$ is given by the brackets:
 \begin{eqnarray*}
q_{1}(v_{1})&=&\mathrm{d}(v_{1})\\
q_{n}({v_{1},\ldots, v_{n}})&=& \Phi_{\pi T(v_{1}\otimes\ldots\otimes v_{n-1})} v_{n},
\end{eqnarray*}
for all $v_{1},\ldots, v_{n}\in V$, $n\geq 2$.

In particular, a strict morphism $T=T_{1}:V\to E$ is an non-abelian embedding tensor if and only if
 \begin{align*}
 &l_{1}\smalcirc T=T\smalcirc \mathrm{d}\\
 &l_{n}\left(T(v_{1}),\ldots, T(v_{n})\right)=T\left(\Phi^{n-1,1}\left(T(v_{1}),\ldots, T(v_{n-1});v_{n}\right)\right), %\quad v_{1},\ldots, v_{n}\in V,\, n\geq 2.
 \end{align*}
 and the  descendent Loday $\infty$-structure on $V$ induced by $T$ is 
 \begin{eqnarray*}
q_{1}(v_{1})&=&\mathrm{d}(v_{1})\\
q_{n}({v_{1},\ldots, v_{n}})&=& \Phi^{n-1,1}(T(v_{1}),\ldots, T(v_{n-1}); v_{n}),
\end{eqnarray*}
for all $v_{1},\ldots, v_{n}\in V$, $n\geq 2$.
\end{ex}

\begin{ex}
Let $(E, \brr{\cdot , \cdot }_E)$ and $(V, \brr{\cdot , \cdot }_V)$ be Lie algebras and $\Phi:E\to \Coder\bar S(V)$ the coherent action induced by a coherent Lie representation $\rho:E\to \Der(V)$:
 $$
\brr{\rho_{x}v,w}_{V}=0, \quad v,w\in V,\; x\in E.
$$ 

 A non-abelian embedding tensor $T\equiv\set{T_{k}: T^{k} (V)\to E}_{k\in\Nn}$ must satisfy

{ \small{
\begin{eqnarray*}
\lefteqn{ \sum_{i=1}^{n-1}\sum_{\sigma\in Sh(i,n-1-i)}\epsilon(\sigma)\brr{T_{\bullet}(v_{\sigma(1)}, \ldots, v_{\sigma(i)}),T_{\bullet}(v_{\sigma(i+1)}, \ldots, v_{\sigma(n-1)},v_{n})}_{E}=}\\
&=&\sum_{k=2}^{n}\sum_{i=0}^{k-2}\sum_{\sigma\in Sh(i,k-1-i)}\epsilon(\sigma) (-1)^{|v_{\sigma(1)}|+\ldots+|v_{\sigma(i)}|}\\
&&T_{\bullet}\left(v_{\sigma(1)}, \ldots, v_{\sigma(i)},\rho_{T_{\bullet}(v_{\sigma(i+1)}, \ldots, v_{\sigma(k-1)})} v_{k}, v_{k+1},\ldots, v_{n} \right)\\
&& + \sum_{j=2}^{n}\sum_{{i=1}}^{j-1} (-1)^{|v_{i}|(\sum_{k=i}^{j-1}|v_{k}|)+\sum_{k=1}^{j-1}|v_{k}|}T_{n-1}(v_{1},\ldots , \widehat{v_{i}},\ldots,v_{j-1},\brr{v_{i},v_{j}}_{V}, v_{j+1},\ldots, v_{n}).
 \end{eqnarray*}
 }}
 
% \
% 
% %%Comentário%%%
% \comm{ou colocar esta equa\c cão
% \begin{align*}
% &\brr{T_{\bullet}(v_{(1)}),T_{\bullet}(v_{(2)},v_{n})}_{E}=\sum_{k=2}^{n} (-1)^{|v_{(1)}|}T_{\bullet}\left(v_{(1)},\rho_{T_{\bullet}(v_{(2)})} v_{k}, v_{k+1},\ldots, v_{n} \right)\\
% &\quad + \sum_{j=2}^{n}\sum_{{i=1}}^{j-1} (-1)^{|v_{i}|(\sum_{k=i}^{j-1}|v_{k}|)+\sum_{k=1}^{j-1}|v_{k}|}T_{n-1}(v_{1},\ldots , \widehat{v_{i}},\ldots,v_{j-1},\brr{v_{i},v_{j}}_{V}, v_{j+1},\ldots, v_{n}), \end{align*}
% for all $v_{1},\ldots, v_{n}\in V.$
%On the left hand side, $v_{(1)}\otimes v_{{(2)}}$ stands for $\Delta^{c}(v_{1}\otimes \ldots\otimes v_{n-1})$ while on the right hand side, for each $k$, $v_{(1)}\otimes v_{{(2)}}$stands for $\Delta^{c}(v_{1}\otimes \ldots\otimes v_{k-1})$.
%
% }
%%%%%%%%%%%%
 
If $T:V\to E$ is a strict morphism, we recover the definition of a non-abelian embedding tensor with respect to a coherent action in the sense of \cite{TS2023}, because for $n=2$ the last equation reduces to
 $$
 \brr{T(v_{1}), T(v_{2})}_{E}=T\left( \rho_{T(v_{1})} v_{2}\right) + T\brr{v_{1},v_{2}}_{V},\quad v_{1},v_{2}\in V.
 $$

\end{ex}

%%%%%%%%%%%%%%%%%%%%%%%%%%%%%%%%%%%
% {Embedding tensors of the adjoint representation} %%
%%%%%%%%%%%%%%%%%%%%%%%%%%%%%%%%%%%
\subsection{Non-abelian embedding tensors of the adjoint representation}
Example \ref{ex:embedding:tensor:representation} can be adapted to the adjoint representation of a Lie $\infty$-algebra $(E,M_{E}\equiv\set{l_{k}}_{k\in\Nn})$.

For the adjoint representation, $T\equiv\set{T_{k}:T(E)\to E}_{k\in\Nn}$ is an non-abelian embedding tensor if and only if
 \begin{eqnarray*}
 &&l_{1}\smalcirc T_{1}=T_{1}\smalcirc {l_{1}},\\
 \lefteqn{l_{\bullet}T(x_{1}\otimes\ldots\otimes x_{n})
 = \sum_{i=1}^n  (-1)^{(|x_{{1}}|+\ldots+ |x_{{i-1}}|)} T_n(x_1,\ldots, l_1(x_i),\ldots, x_n)}\\
&& +\sum_{k=2}^{n}\!\!\!\!\!\sum_{\substack{i=0\\ \sigma\in Sh(i,k-i-1)}}^{k-2}\!\!\!\!\!\!\!\!\!\!\epsilon(\sigma)(-1)^{(|x_{{\sigma(1)}}|+\ldots+ |x_{{\sigma(i)}}|)} \\
 &&T_{\bullet}\left(x_{{\sigma(1)}}, \ldots, x_{{\sigma(i)}},
 l_{\bullet}(T(x_{{\sigma(i+1)}}\otimes \ldots\otimes x_{{\sigma(k-1)}} ), x_{{k}}), 
 x_{{k+1}}, \ldots, x_{{n}}\right).
 \end{eqnarray*} 
 
 The descendent Loday $\infty$-structure on $E$ defined by the non-abelian embedding tensor $T$ is given by the brackets:
 \begin{eqnarray*}
q_{1}(x_{1})&=&l_{1}(x_{1})\\
q_{n}({x_{1},\ldots, x_{n}})&=& l_{\bullet}(T(x_{1}\otimes\ldots\otimes x_{n-1}), x_{n}),
\end{eqnarray*}
for all $x_{1},\ldots,x_{n}\in E$, $n\geq 2$.

In particular, if we ask $T:E\to E$ to be strict we must have
 \begin{eqnarray}
 l_{1}\smalcirc T&=&T\smalcirc l_{1}\nonumber \\
 l_{n}\left(T(x_{1}),\ldots, T(x_{n})\right)&=&T\left(l_{n}\left(T(x_{1}),\ldots,T(x_{n-1}),x_{n}\right)\right), \quad  \label{eq:strict:nonalbelian:embedding:adjoint:rep}
 \end{eqnarray}
 and the Loday $\infty$-brackets are:
 \begin{eqnarray*}
q_{1}(x_{1})&=&l_{1}(x_{1})\\
q_{n}({x_{1},\ldots, x_{n}})&=& l_{\bullet}(T(x_{1}),\ldots, T(x_{n-1}), x_{n}),
\end{eqnarray*}
for all $x_{1},\ldots,x_{n}\in E$, $n\geq 2$.

 \begin{rem}
The identity map is obviously a non-abelian embedding tensor for the adjoint representation and, as it should be, the ``new'' Loday brackets induced by the identity map are the original symmetric brackets that define the Lie $\infty$-structure. 
\end{rem}

\begin{rem}
    Notice that the brackets $l_{\bullet}$ are symmetric, hence, for every $n> 1$ and each $i=1,\ldots, n$, we have 
 $$
 l_{n}\left(T(x_{1}),\ldots, T(x_{n})\right)=T\left(l_{n}(T(x_{1}),\ldots, x_{i},\ldots ,T(x_{n}))\right), \quad x_{1},\ldots, x_{n}\in E.
 $$
\end{rem}

\begin{rem}
Let $T:E\to E$ be a strict comorphism such that $l_{1}\smalcirc T=T\smalcirc l_{1}$ and $\im(T)\subset \Ker (\ad)$. Then, $T$ is an embedding tensor for the adjoint representation but the new Loday $\infty$-algebra is abelian, i.e. $q_{k}=0$, for all $k>1$.
\end{rem}

Equations (\ref{eq:strict:nonalbelian:embedding:adjoint:rep}) yield that, for the adjoint representation, the composition of strict non-abelian embedding tensors is a strict non-abelian embedding tensor, therefore:

\begin{prop}
    The space of strict non-abelian embedding tensors  with respect to the adjoint representation is  an associative algebra with unit, called the \textbf{strict embedding algebra}.
\end{prop}

Let us call  \emph{centroid of the Lie $\infty$-algebra} $(E, \set{l_k}_{k\in\Nn})$ the set of strict comorphisms $F:E\to E$ such that
 \begin{eqnarray*}
l_{1}\smalcirc F&=&F\smalcirc l_{1}\nonumber\\
\ad_{x}\smalcirc F&=&F\smalcirc \ad_{x}, \quad x\in \bar S(E).%\label{eq:centroid:two}
\end{eqnarray*}

\begin{cor} Let $E$ be a Lie $\infty$-algebra. Then the centroid of $E$ is a subalgebra (with unit) of the strict  embedding algebra.  

Moreover, any surjective  strict non-abelian embedding tensor is an element of the centroid and the group of invertible elements of the centroid coincides with the group of invertible elements of the strict embedding algebra. 
    %$$\mbox{centroid of $E$}\subset \set{\mbox{non-abelian embedding tensors of $E$}}$$
\end{cor}

\subsection{Strict non-abelian embedding tensors of the adjoint action}

The adjoint action $\ad: (E,M_{E}\equiv\set{l_k}_{k\in\Nn})\to (\Coder(\bar S(E))[1], \partial_{M_{E}}, \brr{\cdot , \cdot })$ is a special representation of $E$ on the cochain complex $(\bar S(E), M_{E})$.

A strict non-abelian embedding tensor of this representation is a strict comorphism $T:\bar S(E)\to E$ satisfying
 \begin{eqnarray*}
 l_{1}\smalcirc T(v_{1})&=&T\smalcirc {M_{E}}(v_{1})\\
 l_{n}(T(v_{1}),\ldots, T(v_{n}))
 &=& 
T\left(
 \ad_{ T(v_{{1}})\cdot\ldots\cdot T(v_{{n-1}} )} v_{{n}}\right),
 \end{eqnarray*} 
for all $v_{1},\ldots, v_{n}\in \bar S(E)$.

In particular, for $v_{1},\ldots, v_{n}\in E$, we have 
 \begin{eqnarray*}
 l_{1}\smalcirc T(v_{1})&=&T\smalcirc l_{1}(v_{1})\\
 l_{n}(T(v_{1}),\ldots, T(v_{n}))
 &=& 
T\left(
 l_{n} ( T(v_{{1}}),\ldots, T(v_{{n-1}} ), v_{{n}})\right)
 \end{eqnarray*} 
and we conclude that $T_{|_{E}}:E\to E$ is a strict non-abelian embedding tensor for the adjoint representation.

%%%%%%%%%%%%%%%%%%%%%%%%%%%%%%%%%%%%%%%%%%%%%%%%%%%%%%%%%%%%
\subsection{Deformations of non-abelian embedding tensors}

Let $T$ be a non-abelian embedding tensor with respect to the coherent Lie $\infty$-action $\Phi:(E, M_{E})\to (\Coder(\bar S(V))[1], \partial_{M_V}, \brr{\cdot ,\cdot })$. Since $T$ defines (and is defined) by $\mathfrak{t}$, a Maurer-Cartan element of the Lie $\infty$-algebra $\mathfrak{h}$, we have two different $V$-data $(\Coder T^{Z}(E\oplus V), \mathfrak{h}, \mathcal{P}, e^{\brr{\, , \, \mathfrak t}}(Q))$ and $(\Coder T^{Z}(E\oplus V), \mathfrak{h}, \mathcal{P}\smalcirc e^{\brr{\, , \, \mathfrak t}}, Q)$, both of them inducing the same twisting Lie $\infty$-structure in $\mathfrak{h}$:
\begin{equation*}
\partial_{k}^{T}(\mathfrak{t}_{1},\ldots, \mathfrak{t}_{k})=\sum_{i\geq 0}\frac{1}{i!}\partial_{k+i}(\mathfrak{t},\ldots, \mathfrak{t},\mathfrak{t}_{1},\ldots,\mathfrak{t}_{k}), \quad \mathfrak{t}_{1},\ldots,\mathfrak{t}_{k}\in\mathfrak{h}, \quad k\geq 1.
\end{equation*}

Let $T':T^{Z} (V)\to T^{Z} (E)$ be another comorphism and consider $\mathfrak{t'}$ the coderivation of $T^{Z}(E\oplus V)$ defined by $T'$. 
By definition, $T+T'$ is an embedding tensor if and only if $\mathfrak t + \mathfrak t'$ is a Maurer-Cartan element of $\mathfrak h$:
$$
\sum_{k\geq 1}\frac{1}{k!}\partial_{k}(\mathfrak t + \mathfrak t',\ldots,\mathfrak t + \mathfrak t') =0.
$$

Taking into account that $\mathfrak t\in \mathrm{MC}(\mathfrak{h})$, last equation can be rewritten \cite{FZ2015} as 
$$
\partial_{1}^{T}(\mathfrak{t}') + \frac{1}{2}\partial_{2}^{T}(\mathfrak{t}',\mathfrak{t}') + \ldots + \frac{1}{k!}\partial_{k}^{T}(\mathfrak{t}',\ldots, \mathfrak{t}')+\ldots =0.
$$
Therefore $T+T'$ is a non-abelian embedding tensor with respect to the same Lie $\infty$-action as $T$ if and only if $T'$ is a Maurer-Cartan element of the deformed Lie $\infty$-algebra $\mathfrak{h}^{T}=(\mathfrak{h}, \partial^{T})$.

\begin{prop}
The Lie $\infty$-algebra $\mathfrak{h}^{T}=(\mathfrak{h}, \partial^{T})$ controls the deformations of $T$. \end{prop}

\begin{defn}
The cohomology of the 
 cochain complex $(\bar S(\mathfrak{h}), \partial^{T})$ is called the \textbf{cohomology of the non-abelian embedding tensor $T$}.
\end{defn}

\section{Conclusions}

In \cite{TS2023} a definition of non-abelian embedding tensors on a Lie algebra with respect to a coherent  action was given. 
In this work we give a possible  generalization of this definition and of some  results in \cite{STZ21,TS2023} to the Lie $\infty$-setting.   We propose a definition of coherent  Lie $\infty$-actions and, for these actions,  we construct   the non-abelian hemisemidirect product and study non-abelian embedding tensors.
The question of whether such a generalization could be carried out  was posed by Y. Sheng  to one of the co-authors
 during a Colloquium, last year. There are still open questions  concerning this generalization.

Firstly, we can ask about  the existence of a  Lie $\infty$-representation/action   of the Lie $\infty$-algebra $\mathrm{Prim}(V)$  on $V$, such that the projection $T(V)\to \mathrm{Prim}(V) $ is an embedding tensor.
Secondly, since there is a straight  relation between non-abelian embedding  tensors for Lie algebras and Leibniz-Lie algebras \cite{TS2023}, we wonder what is the algebraic structure that is related to the embedding tensors we consider.  What about cohomology of this algebraic structure? We think the answer to these questions may go through  understanding what a Loday $\infty$-action should be and 
these problems will deserve our attention in a future work.


\begin{thebibliography}{99}

\bibitem{A2000} Aguiar, M. (2000). Pre-Poisson Algebras. \emph{Letters in Mathematical Physics}. {54}: 263–277.

\bibitem{AP2010} Ammar, M., Poncin, N. (2010) Coalgebraic approach to the Loday infinity category, stem differential for $2n$-ary graded and homotopy algebras. \emph{Ann. Inst. Fourier}. {60}: 355--387.


\bibitem{B97} Balavoine, D. (1997).  Deformation of algebras over a quadratic operad. 
In:  Loday, J.-L.,  Stasheff, J.,   Voronov, A. A, eds. \emph{Operads: Proc. Renaissance Conf., Contemp. Math.}, vol. 202, AMS, Providence, RI (1997), pp. 207-234



\bibitem{BH2020} Bonezzi, R. (2020).  Hohm, O.: Leibniz Gauge theories and infinity structures. \emph{Commun. Math. Phys.}  {377}: 2027--2077.


\bibitem{CC2022} Caseiro, R., Nunes da Costa, J. (2022). $\mathcal{O}$-Operators on Lie $\infty$-algebras with respect to Lie $\infty$-actions. \emph{Communications in Algebra}. {50} ({7}): 3079--3101.

\bibitem{FZ2015} Fr\'egier, Y., Zambon, M. (2015). Simultaneous deformations of algebras and morphisms via derived brackets. \emph{Journal of Pure and Applied Algebra}. {219}: 5344--5362.

\bibitem{G} Getzler, E. (2009). Lie theory for nilpotent $L_\infty$-algebras. \emph{Ann. Math.} {170}(2): 271--301.

\bibitem{GHP2014} Greitz, J., Howe, P., Palmkvist, J. (2014). The tensor hierarchy simplified. \emph{Classical Quantum Gravity}. {31}: 087001. 

\bibitem{KW2001} Kinyon, M., Weinstein, A. (2001). Leibniz algebras, Courant algebroids and multiplications on reductive homogeneous spaces. \emph{Amer. J. Math.} {123}: 525--550. 

\bibitem{KS2020} Kotov, A., Strobl, T. (2020). The embedding tensor, Leibniz-Loday algebras, and their higher Gauge theories. \emph{Comm. Math. Phys.} {376}: 235--258.



\bibitem{LSS2014} Lavau, S., Samtleben H., Strobl, T. (2014). Hidden $Q$-structure and Lie $3$-algebra for non-abelian superconformal models in six dimensions. \emph{J. Geom. Phys.}  {86}: 497-533.


\bibitem{L2019} Lavau, S. (2019). Tensor hierarchies and Leibniz algebras. \emph{J. Geom. Phys.}  {144}: 147--189.

\bibitem{LP2020} Lavau, S., Palmkvist, J. (2020). Infinity-enhancing of Leibniz algebras.  \emph{Letters in Math. Phys.} {110}: 3121--3152.

\bibitem{LS2023} Lavau, S., Stasheff, J. (2023). From Lie algebra crossed modules to tensor hierarchies. \emph{Journal of Pure Applied Algebra}.  {227}: 107311.

\bibitem{MZ} Mehta, R., Zambon, M. (2012). $L_\infty$-algebra actions. \emph{Differential Geometry and its Applications}. {30}: 576--587.

%\bibitem{NS2001} Nicolai, H., Samtleben, H.: Maximal gauged supergravity in three dimensions. \emph{Phys. Rev. Lett.} {(2001)}, \textbf{86}, 1686--1689.

\bibitem{P2014} Palmkvisk, J. (2014). The tensor hierarchy algebra. \emph{J. Math. Phys.}  {55}: 011701.

\bibitem{R03} Reutenauer, C. (2003). Free Lie algebras. In: Hazewinkel, M., ed. \emph{Handbook of Algebras}
Vol. 3,     North-Holland, Amesterdam, pp. 887--903.





\bibitem{R95} Rota, G.-C. (1995). Baxter operators, an introduction. In: King, J. P. S., ed. \emph{Gian-Carlo Rota on Combinatorics: Introductory Papers and Commentaries}, Birkauser, Boston 1995.

\bibitem{STZ21} Sheng, Y., Tang, R., Zhu, C. (2021). The controlling L$_\infty$-algebra, cohomology and homotopy of embedding tensors and Lie-Leibniz triples. \emph{Commun. Math. Phys.} {386}: 269--304.

\bibitem{TS2023} Tang, R, Sheng, Y. (2023). Non-abelian embedding tensors. \emph{Lett. Math. Phys.}  {113}: 14.

\bibitem{U2011} Uchino, K. (2011). Derived brackets and sh Leibniz algebras. \emph{Journal of Pure and Applied Algebra}. {215}(5): 1102--1111.


\bibitem{V05} Voronov, T. (2005). Higher derived brackets and homotopy algebras. \emph{J. Pure Appl. Algebra}. {202}(1-3): 133--153.
\end{thebibliography}
\end{document}